\documentclass[12pt,twoside]{article}

\usepackage{amsmath, amsthm, amssymb,bm}

\numberwithin{equation}{section}

\newcommand{\bC}{\mathbb{C}}

\newcommand{\bE}{\mathbb{E}}

\newcommand{\mcal}[1]{\mathcal{#1}}

\def\({ \left( }
\def\){ \right)}

\def\rmU{\mathrm{U}}
\def\rmO{\mathrm{O}}
\def\rmSp{\mathrm{Sp}}

\def\rmU{\mathrm{U}}

\DeclareMathOperator{\Tr}{Tr}

\def\({ \left( }
\def\){ \right)}

\def\trans#1{\mathord{#1}^{\mathrm{T}}}
\def\dual#1{\mathord{#1}^{\mathrm{D}}}

\theoremstyle{plain}
\newtheorem{thm}{Theorem}[section]
\newtheorem{prop}[thm]{Proposition}
\newtheorem{lem}[thm]{Lemma}

\theoremstyle{definition}

\newtheorem{example}{Example}[section]
\newtheorem{remark}{Remark}[section]
\theoremstyle{conjecture}

\theoremstyle{problem}

\title{\textbf{Weingarten calculus for matrix ensembles
associated with compact symmetric spaces}}
\author{\textsc{Sho Matsumoto}}

\date{\empty}

\begin{document}

\maketitle

\begin{abstract}
A method for computing integrals of polynomial functions on
compact symmetric spaces is given.
Those integrals are expressed as sums of functions on symmetric groups.
\\

\noindent
\emph{Keywords}: Weingarten calculus, compact symmetric space, matrix elements. \\
Mathematics Subject Classification 2010:
15B52, 
28C10, 
43A85  

\end{abstract}


\section{Introduction}

Let $G$ be 
the unitary group $\mathrm{U}(n)$,
orthogonal group $\mathrm{O}(n)$, 
or (compact) symplectic group $\rmSp(2n)$,
equipped with its Haar probability measure, 
and suppose that $G$ is realized as a matrix group. 
Consider a random matrix $X=(x_{ij})$ picked up from $G$.
Moments of matrix elements 
\begin{equation} \label{eq:moments-G}
\begin{array}{rl}
\bE[ x_{i_1 j_1} \cdots x_{i_k j_k} \overline{x_{i_1' j_1'} \cdots x_{i_l' j_l'}}]
\qquad & \text{for $G=\mathrm{U}(n)$},  \\
\bE[ x_{i_1 j_1} \cdots x_{i_k j_k} ]
\qquad & \text{for $G=\mathrm{O}(n)$ or $\mathrm{Sp}(2n)$} 
\end{array}
\end{equation}
have been one of main interesting objects in random matrix theory.
Weingarten \cite{Weingarten} tried to understand the moments \eqref{eq:moments-G}
for large $n$.
Collins \cite{Collins} established a method of computations of \eqref{eq:moments-G}
for $\mathrm{U}(n)$ with all finite $n$ and
called his method the {\it Weingarten calculus}.
The moments are expressed as sums in terms of  
class functions $\mathrm{Wg}^{\mathrm{U}}(\cdot;n)$
on the symmetric group $S_{k}$.
The function $\mathrm{Wg}^{\mathrm{U}}(\cdot;n)$, 
called the unitary Weingarten function,
has rich combinatorial structures involving Jucys-Murphy elements \cite{MN,Novak}.
The Weingarten calculus for $\mathrm{O}(n)$ was constructed in \cite{CollinsSniady,
CM}, and  the corresponding Weingarten functions,
called orthogonal Weingarten functions, 
are $H_k$-biinvariant functions on $S_{2k}$, where
$H_k$ is the hyperoctahedral subgroup of $S_{2k}$.
The orthogonal Weingarten function also bears
nice combinatorics \cite{M-JMOrtho,ZJ}.
The Weingarten calculus for $\mathrm{Sp}(2n)$ was mentioned in \cite{CollinsSniady,
CollinsStolz}. See also a recent preprint \cite{Dah}, in which
Brownian motions on classical groups are studied.
In the last decade,
Weingarten calculus has been widely applied: 
Harish-Chandra--Itzykson--Zuber integrals \cite{Collins, GGJ},
quantum information theory \cite{CN, CFN}, 
designs \cite{Scott}, and statistics \cite{CMS}.

In his pioneering works for statistical mechanics,
Dyson \cite{Dyson1, Dyson2, Dyson3} introduced three important classes of random unitary matrices
as modifications of Gaussian matrix ensembles.
These unitary matrix ensembles are known as
{\it circular orthogonal/unitary/symplectic ensembles} (COE/CUE/CSE). 
Whereas the CUE is nothing but the unitary group equipped with the Haar measure,
the COE and CSE are not Lie groups but compact symmetric spaces.

Generally, a classical compact symmetric space is of the form $G/K$, where
$G$ is a unitary, orthogonal, or symplectic group, and $K$ is a closed subgroup of $G$ consisting of
fixed-points of an involution $\Omega$ on $G$. 
Let $\mcal{S}$ be the image of the map $G \ni g \mapsto \Omega(g)^{-1} g$.
We thus obtain a matrix realization of $G/K$:
$$
G/K \simeq \mcal{S}; \qquad G \ni g \mapsto \Omega(g)^{-1} g \in \mcal{S},
\qquad K= \{ h \in G \ | \ \Omega(h)=h\}.
$$
Moreover,
a $G$-invariant probability measure on $\mcal{S}$ is induced from 
the Haar measure on $G$ via this map.
Then the matrix space $\mcal{S}$ equipped with the probability measure 
is the matrix ensemble associated with $G/K$.
In the case of the COE, $G/K=\rmU(n)/\rmO(n)$,
$\Omega(g)=\overline{g}$ (the complex-conjugation of $g$),
and $\mcal{S}$ is the set of $n \times n$ symmetric unitary matrices.
Cartan  \cite{Cartan} classified classical compact symmetric spaces into 
seven infinite series labeled by A~I, A~II, A~III, BD~I, C~I, C~II,  and D~III
(see Figure \ref{ListCSS}).
Thus we obtain seven series of random matrix ensembles
associated with these compact symmetric spaces. 
Eigenvalue distributions for them have been studied, 
see, e.g., \cite{Du}, \cite{CaM}, and \cite[\S 3.7.2 and 3.7.3]{F}.
In contrast, distributions for those matrix elements were not studied so much.
Ones can see a small work for only classes A~I and A~II in
\cite[\S 2.3.2]{F}.


In this article,
we will construct Weingarten calculus for random matrices
associated with each class $\mcal{C}=$ A~I, A~II,... in Figure \ref{ListCSS}.
A development in a similar direction
can be seen in \cite{CollinsStolz}, but 
only asymptotic behaviors of two-degree moments were observed in that 
article.
As in the case of classical groups $\rmU(n), \rmO(n),\rmSp(2n)$,
we will express the moments of random matrix elements as 
sums over symmetric groups.
To do it,
we need a distinguishing  Weingarten function $\mathrm{Wg}^{\mcal{C}}$
for each $\mcal{C}$.

This article is organized as follows.
In Section 2, we review the Weingarten calculus for $\rmU(n)$, $\rmO(n)$, and 
$\rmSp(2n)$, developed in \cite{Collins, CM, CollinsSniady, CollinsStolz}.
Especially, the case for symplectic groups is discussed in detail.
It is described in \cite{CollinsStolz, Dah},
however, unlike their descriptions, we introduce the symplectic Weingarten function 
in terms of twisted zonal spherical functions 
for the twisted  Gelfand pair $(S_{2k},H_k,\epsilon|_{H_k})$.
In Section 3, we discuss the Weingarten calculus for random matrices from 
class A~I and A~II,
which are the 
COE and CSE, respectively. 
The COE case ($\mcal{C}=$ A I) is already given in \cite{M-COE, M-COE-single} in detail
but we present it here again for readers' convenience.
In Section 4 and 5, we discuss the Weingarten calculus 
for chiral ensembles (A~III, BD~I, C~II)
and Bogoliubov-de~Gennes (BdG) ensembles (D~III, C~I),
respectively. 
In the final section, we give a short conclusion for new  Weingarten functions.

\begin{figure} \centering
 \begin{tabular}{|l|l|l|} \hline
 Class $\mcal{C}$ & Symmetric spaces & Random matrices \\ \hline
 A I & $\mathrm{U}(n)/\mathrm{O}(n)$ & circular orthogonal ensemble (COE)  \\ 
 A II & $\mathrm{U}(2n)/\mathrm{Sp}(2n)$ & circular symplectic ensemble (CSE) \\ \hline
 A III & $\mathrm{U}(n)/\mathrm{U}(a) \times \mathrm{U}(b) \quad (a+b=n)$ &  \\
 BD I & $\mathrm{O}(n)/\mathrm{O}(a) \times \mathrm{O}(b) \quad (a+b=n)$ & chiral ensemble \\
  C II & $\mathrm{Sp}(2n)/\mathrm{Sp}(2a) \times \mathrm{Sp}(2b) \quad (a+b=n)$ &  \\
 \hline
  D III & $\mathrm{O}(2n)/\mathrm{U}(n) $ & Bogoliubov-de Gennes (BdG) ensemble \\
  C I & $\mathrm{Sp}(2n)/\mathrm{U}(n) $ &  \\
  \hline
 \end{tabular}
 \caption{{\bf List of classical compact symmetric spaces.}
 Remark that we consider the full groups $\rmU(n)$ and $\rmO(n)$
rather than special ones $\mathrm{SU}(n)$ and $\mathrm{SO}(n)$,
the latter groups of which are usually used in Lie theory.
 }
 \label{ListCSS}
\end{figure}

\section{Weingarten calculus for classical groups}

\subsection{Weingarten calculus for unitary groups}

In this subsection, we review the Weingarten calculus for unitary groups.
See \cite{Collins,CollinsSniady,CMS} for details.

\subsubsection{Partitions and cycle-types}

A partition $\lambda=(\lambda_1,\lambda_2,\dots,\lambda_l)$
of a positive integer $k$ is 
a weakly decreasing sequence of positive integers with
$|\lambda|=\sum_{i=1}^l \lambda_i=k$.
We write  $\ell(\lambda)$ for the length $l$ of $\lambda$.
When $\lambda$ is a partition of $k$, we write  $\lambda \vdash k$.

Let $S_k$ be the symmetric group acting on $[k]=\{1,2,\dots,k\}$.
A permutation $\sigma \in S_k$ is usually expressed in the two-array notation
$\(\begin{smallmatrix} 1 & 2 & \cdots & k \\ \sigma(1) & \sigma(2) & \cdots &
\sigma(k) \end{smallmatrix}\)$,
but  transpositions are often written as $(i \ j)$ shortly.
Denote by $\mathrm{id}_k$ the identity permutation in $S_k$.
If a permutation $\sigma$ in $S_k$ is decomposed into disjoint cycles with lengths
$\mu_1 \ge \mu_2 \ge \cdots \ge \mu_l$, then 
the partition $\mu=(\mu_1,\mu_2,\dots,\mu_l)$ of $k$ is called 
the {\it cycle-type} of $\sigma$. 
For example,
the cycle-type of the permutation
$\( \begin{smallmatrix} 1 & 2 & 3 & 4 & 5 & 6 \\
3 & 1 & 2 & 4 & 6 & 5 \end{smallmatrix}\)$ in $S_6$ 
is $(3,2,1)$.

\subsubsection{Unitary Weingarten functions}

Let $L(S_k)$ be the algebra of complex-valued functions on $S_k$ with convolution
$$
(f_1*f_2)(\sigma)= \sum_{\tau \in S_k} f_1(\tau) f_2(\tau^{-1} \sigma) 
\qquad (f_1,f_2 \in L(S_k), \ \sigma \in S_k).
$$
The identity element in the algebra $L(S_k)$ is
 the Dirac function $\delta_{\mathrm{id}_k}$.

Let $\mcal{Z}(L(S_k))$ be the center of $L(S_k)$:
$\mcal{Z}(L(S_k))=\{ h \in L(S_k) \ | \ h * f =f*h \ (f \in L(S_k))\}$.
For a complex number $z$, we define the element $T^{\mathrm{U}}(\cdot ;z)$ in
$\mcal{Z}(L(S_k))$ by
$$
T^{\mathrm{U}}(\sigma;z)= z^{\ell(\mu)}\qquad (\sigma \in S_k),
$$
where $\mu$ is the cycle-type of $\sigma$.
Note that the upper index $\mathrm{U}$ stands for the unitary group.
The class function $T^{\mathrm{U}}(\cdot;z)$ can be expanded
in terms of irreducible characters $\chi^\lambda$ of $S_k$ as follows:
$$
T^{\mathrm{U}}(\cdot;z)=
\frac{1}{k!} \sum_{\lambda \vdash k} f^\lambda C_\lambda(z) \chi^\lambda,
$$
where $f^\lambda:=\chi^\lambda (\mathrm{id}_k)$ 
is the dimension of the irreducible representation associated with $\lambda$ and 
$C_\lambda(z)$ is the polynomial in $z$ given by
$$
C_\lambda(z)= 
\prod_{(i,j) \in \lambda}  (z+j-i).
$$
Here $(i,j) \in \lambda$ stands for 
$1 \le i \le \ell(\lambda), \, 1 \le j \le \lambda_i$.
In other words, the $(i,j)$ are coordinates of the Young diagram of $\lambda$.

The {\it unitary Weingarten function} $\mathrm{Wg}^{\mathrm{U}}(\cdot;z)$ on $S_k$ is 
defined by
\begin{equation} \label{def-uWg}
\mathrm{Wg}^{\mathrm{U}}(\cdot;z) = \frac{1}{k!}
\sum_{\begin{subarray}{c} \lambda \vdash k \\
C_\lambda(z) \not=0 \end{subarray}}
\frac{f^\lambda}{C_\lambda(z)} \chi^\lambda,
\end{equation}
summed over all partitions $\lambda$ of $k$ with $C_\lambda(z) \not=0$. 
It is the pseudo-inverse element of $T^{\mathrm{U}}(\cdot;z)$, i.e., the unique element in $\mcal{Z}(L(S_k))$ satisfying
$$
T^{\mathrm{U}}(\cdot;z)* \mathrm{Wg}^{\mathrm{U}}(\cdot;z)* T^{\mathrm{U}}(\cdot;z) =T^{\mathrm{U}}(\cdot;z).
$$
In particular, unless $z \in \{0,\pm 1, \pm 2, \dots, \pm(k-1)\}$,
functions $T^{\mathrm{U}}(\cdot;z)$ and 
$\mathrm{Wg}^{\mathrm{U}}(\cdot;z)$ are inverse of each other and 
satisfy $T^{\mathrm{U}}(\cdot;z)* \mathrm{Wg}^{\mathrm{U}}(\cdot;z)=\delta_{\mathrm{id}_k}$.

\subsubsection{Integrals on unitary groups}

The unitary group $\mathrm{U}(n)$ is
$\mathrm{U}(n) =\{ U \in \mathrm{GL}(n,\mathbb{C}) \ | \ 
U U^* = I_n\}$,
which has the Haar probability measure.
Here $U^*:=\overline{\trans{U}}$ is the adjoint matrix of $U$
and $I_n=(\delta_{ij})_{1 \le i,j \le n}$ is the $n \times n$ identity matrix.
We will also write $U$ for a random element of $\rmU(n)$, with distribution given by Haar measure.

\begin{thm}[Weingarten calculus for unitary groups \cite{Collins, CollinsSniady}]
\label{thm:WgCal-Unitary}
Let $U=(u_{ij})_{1 \le i,j \le n}$ be an $n \times n$ Haar 
unitary matrix.
For four sequences $\bm{i}=(i_1,\dots,i_k)$, $\bm{j}=(j_1,\dots,j_k)$,
$\bm{i}'=(i_1',\dots,i_k')$, $\bm{j}'=(j_1',\dots,j_k')$
of positive integers in $[n]$, we have
$$
\bE[ u_{i_1 j_1} \cdots u_{i_k j_k}
\overline{ u_{i_1'j_1'} \cdots u_{i_k' j_k'}}]
= \sum_{\sigma, \tau \in S_k} \delta_\sigma (\bm{i},\bm{i}') \delta_\tau
(\bm{j},\bm{j}') \mathrm{Wg}^{\mathrm{U}}(\sigma^{-1}\tau;n).
$$
Here $\delta_{\sigma}(\bm{i},\bm{i}')$ is defined by
\begin{equation} \label{eq:delta-ft}
\delta_{\sigma}(\bm{i},\bm{i}') = \prod_{s=1}^k \delta_{i_{\sigma(s)},i'_{s}}.
\end{equation}
If $k \not=l$, then
$\bE[ u_{i_1 j_1} \cdots u_{i_k j_k}
\overline{ u_{i_1'j_1'} \cdots u_{i_l' j_l'}}]$ vanishes
for any indices $i_1,\dots,i_k$, $j_1,\dots,j_k$, $i_1',\dots,i'_l$,
$j_1',\dots,j_l'$.
\end{thm}

\begin{example}
$\mathrm{Wg}^{\mathrm{U}} ( \mathrm{id}_1;n)= \frac{1}{n}$;
$$
\mathrm{Wg}^{\mathrm{U}} ( \mathrm{id}_2;n)= \frac{1}{(n+1)(n-1)};
\qquad 
\mathrm{Wg}^{\mathrm{U}} ( (1 \ 2) ;n)= \frac{-1}{n(n+1)(n-1)}.
$$
We can see more examples in \cite{Collins}.
\end{example}

\begin{remark}
In the sum of \eqref{def-uWg},
we can drop the restriction for $\lambda$.
Even if we replace the definition by 
$$
\mathrm{Wg}^{\mathrm{U}}(\cdot;z) = \frac{1}{k!}
\sum_{\lambda \vdash k}
\frac{f^\lambda}{C_\lambda(z)} \chi^\lambda,
$$
which is a rational function in $z \in \bC$ with finitely many poles,
then Theorem \ref{thm:WgCal-Unitary} remains true
after cancellations of poles.
This enables us to ignore the restriction $C_\lambda(z) \not=0$ in latter sections.
See \cite[Proposition 2.5]{CollinsSniady} for details.
\end{remark}

\subsection{Weingarten calculus for orthogonal groups}

In this subsection, we review the Weingarten calculus for orthogonal groups,
developed in 
\cite{CollinsSniady, CM,CMS}.
The theory of zonal spherical functions for finite Gelfand pairs are seen in
\cite[Chapter~VII]{Mac}.

\subsubsection{Hyperoctahedral groups and coset-types}

Let $H_k$ be the hyperoctahedral subgroup of $S_{2k}$ with order $2^k k!$,
generated by
adjacent transpositions $(2i-1 \ 2i)$ $(1 \le i \le k)$ and 
double transpositions
$(2i-1 \ 2j-1)(2i \ 2j)$  $(1 \le i<j \le k)$.
Let $M_{2k}$ be the subset of permutations $\sigma \in S_{2k}$ satisfying
$$
\sigma(2i-1)< \sigma(2i) \ (1 \le i \le k) \qquad 
\text{and} \qquad \sigma(1)< \sigma(3)< \cdots < \sigma(2k-1).
$$ 
The elements $\sigma$ in $M_{2k}$ form the complete set of the representatives of  $S_{2k}/H_k$ and 
are sometimes identified with 
perfect matchings of the forms
$$
\{ \{\sigma(1),\sigma(2)\}, \{\sigma(3), \sigma(4)\},\dots,\{
\sigma(2k-1), \sigma(2k)\}\}
$$
on $[2k]$.

For a permutation $\sigma\in S_{2k}$, we define the graph 
$\Gamma(\sigma)$ as follows.
The vertex set  is $\{1,2,\dots,2k\}$
and the edge set consists of red edges $\{2r-1, 2r\}$
and blue edges $\{\sigma(2r-1),\sigma(2r)\}$, where $r$ runs over $1,2,\dots,k$.
Then all connected components of the graph $\Gamma(\sigma)$ 
have even vertices with numbers $2\mu_1 \ge 2\mu_2 \ge \dots \ge 2\mu_l$.
We call the partition $\mu:=(\mu_1,\mu_2,\dots,\mu_l) \vdash k$
the {\it coset-type} of $\sigma$. 
For example,
the coset-type of the permutation
$\( \begin{smallmatrix} 1 & 2 & 3 & 4 & 5 & 6 \\
3 & 1 & 2 & 4 & 6 & 5 \end{smallmatrix}\)$ in $S_6$ 
is $(2,1)$.
The coset-type distinguishes with double cosets of $H_k$ in $S_{2k}$.
Specifically,
for two permutations $\sigma,\tau$ in $S_{2k}$,
their coset-types coincides if and only if $H_k \sigma H_k=H_k\tau H_k$
(\cite[VII, (2.1)]{Mac}).

For a partition $\mu=(\mu_1,\mu_2,\dots,\mu_l)$ of $k$,
we define the permutation $\sigma_\mu$ in $S_{2k}$ as follows:
For each $r=1,2,\dots,l$,
\begin{equation} \label{eq:sigma-mu}
\begin{array}{l}
\sigma_\mu \( 2\sum_{i=1}^{r-1} \mu_i +1\) = 2\sum_{i=1}^{r-1} \mu_i +1, \\
\sigma_\mu \( 2\sum_{i=1}^{r-1} \mu_i +2\) = 2\sum_{i=1}^{r-1} \mu_i +2 \mu_r, \\\sigma_\mu \( 2\sum_{i=1}^{r-1} \mu_i +p\) = 2\sum_{i=1}^{r-1} \mu_i +p-1
\quad \text{for $p=3,4,\dots, 2\mu_r$}.
\end{array}
\end{equation}
For example, $\sigma_{(3,1)}=
\( \begin{smallmatrix} 1 & 2 & 3 & 4 & 5 & 6 & 7 & 8 \\
1 & 6 & 2 & 3 & 4 & 5 & 7 & 8 \end{smallmatrix}\)$.  
It is easy to see that the permutation $\sigma_\mu$ belongs to $M_{2k}$ and has 
the coset-type $\mu$ and signature $+1$.
In particular, $\sigma_{(1^k)}=\mathrm{id}_{2k}$, the identity permutation in $S_{2k}$.
We often regard $\sigma_\mu$ as a typical permutation of coset-type $\mu$.

\subsubsection{Orthogonal Weingarten functions and Gelfand pairs}

Let $L(S_{2k},H_k)$ be the subspace of all $H_k$-biinvariant functions in $L(S_{2k})$:
$$
L(S_{2k},H_k) = \{ f \in L(S_{2k}) \ | \ 
f(\zeta \sigma)=f(\sigma \zeta)= f(\sigma) \ (\sigma \in S_{2k}, \ \zeta \in H_k)\}.
$$
It is well known that $(S_{2k},H_k)$ is a {\it Gelfand pair}, i.e., 
$L(S_{2k},H_k)$ is a commutative algebra under the convolution 
(\cite[VII, (1.1) and (2.2)]{Mac}).

For two functions $f_1, f_2$ in $L(S_{2k})$, we define 
a new product $f_1 \star f_2$ by 
$$
(f_1 \star f_2)(\sigma)= \sum_{\tau \in M_{2k}} f_1(\tau) f_2(\tau^{-1} \sigma)
\qquad (\sigma \in S_{2k}).
$$

\begin{example}
For $f_1,f_2 \in L(S_4)$,
\begin{align*}
(f_1 \star f_2) \( \begin{smallmatrix} 1 & 2 & 3 & 4 \\
1 & 2 & 3 & 4 \end{smallmatrix}\)
=& f_1\( \begin{smallmatrix} 1 & 2 & 3 & 4 \\
1 & 2 & 3 & 4 \end{smallmatrix}\)
f_2\( \begin{smallmatrix} 1 & 2 & 3 & 4 \\
1 & 2 & 3 & 4 \end{smallmatrix}\) +
f_1\( \begin{smallmatrix} 1 & 2 & 3 & 4 \\
1 & 3 & 2 & 4 \end{smallmatrix}\)
f_2\( \begin{smallmatrix} 1 & 2 & 3 & 4 \\
1 & 3 & 2 & 4 \end{smallmatrix}\)  + 
f_1\( \begin{smallmatrix} 1 & 2 & 3 & 4 \\
1 & 4 & 2 & 3 \end{smallmatrix}\)
f_2\( \begin{smallmatrix} 1 & 2 & 3 & 4 \\
1 & 3 & 4 & 2 \end{smallmatrix}\). \\
(f_1 \star f_2) \( \begin{smallmatrix} 1 & 2 & 3 & 4 \\
1 & 3 & 2 & 4 \end{smallmatrix}\)
=& f_1\( \begin{smallmatrix} 1 & 2 & 3 & 4 \\
1 & 2 & 3 & 4 \end{smallmatrix}\)
f_2\( \begin{smallmatrix} 1 & 2 & 3 & 4 \\
1 & 3 & 2 & 4 \end{smallmatrix}\) +
f_1\( \begin{smallmatrix} 1 & 2 & 3 & 4 \\
1 & 3 & 2 & 4 \end{smallmatrix}\)
f_2\( \begin{smallmatrix} 1 & 2 & 3 & 4 \\
1 & 2 & 3 & 4 \end{smallmatrix}\)  + 
f_1\( \begin{smallmatrix} 1 & 2 & 3 & 4 \\
1 & 4 & 2 & 3 \end{smallmatrix}\)
f_2\( \begin{smallmatrix} 1 & 2 & 3 & 4 \\
1 & 4 & 3 & 2 \end{smallmatrix}\). \\
(f_1 \star f_2) \( \begin{smallmatrix} 1 & 2 & 3 & 4 \\
1 & 4 & 2 & 3 \end{smallmatrix}\)
=& f_1\( \begin{smallmatrix} 1 & 2 & 3 & 4 \\
1 & 2 & 3 & 4 \end{smallmatrix}\)
f_2\( \begin{smallmatrix} 1 & 2 & 3 & 4 \\
1 & 4 & 2 & 3 \end{smallmatrix}\) +
f_1\( \begin{smallmatrix} 1 & 2 & 3 & 4 \\
1 & 3 & 2 & 4 \end{smallmatrix}\)
f_2\( \begin{smallmatrix} 1 & 2 & 3 & 4 \\
1 & 4 & 3 & 2 \end{smallmatrix}\)  + 
f_1\( \begin{smallmatrix} 1 & 2 & 3 & 4 \\
1 & 4 & 2 & 3 \end{smallmatrix}\)
f_2\( \begin{smallmatrix} 1 & 2 & 3 & 4 \\
1 & 2 & 3 & 4 \end{smallmatrix}\). 
\end{align*}
\end{example}

For $f_1,f_2 \in L(S_{2k},H_k)$, we have $f_1 \star f_2 =
(2^k k!)^{-1} f_1 *f_2\in L(S_{2k},H_k)$.
In fact, since $M_{2k}$ is the complete set of representatives of cosets $\sigma H_k$
in $S_{2k}$ and since $f_1,f_2$ are $H_k$-biinvariant, we have
$$
(f_1 *f_2) (\sigma)= \sum_{\tau \in M_{2k}}
\sum_{\zeta \in H_k} f_1 (\tau \zeta) f_2((\tau \zeta)^{-1} \sigma)=
\sum_{\tau \in M_{2k}}
\sum_{\zeta \in H_k} f_1 (\tau) f_2(\tau^{-1} \sigma) =|H_k| 
(f_1 \star f_2)(\sigma).
$$
Hence $L(S_{2k},H_k)$ is a commutative algebra under 
the product $\star$ with the identity element
$$
\bm{1}_k(\sigma)= \begin{cases}
1 & \text{if $\sigma \in H_k$}, \\
0 & \text{if $\sigma \in S_{2k} \setminus H_k$}.
\end{cases}
$$

For each partition $\lambda=(\lambda_1,\lambda_2,\dots) \vdash k$,
we define the {\it zonal spherical function} $\omega^\lambda$
of the Gelfand pair $(S_{2k},H_k)$ by
$$
\omega^\lambda= (2^k k!)^{-1} \chi^{2\lambda} * \bm{1}_k
$$
with $2\lambda=(2\lambda_1,2\lambda_2,\dots)$.
The $\omega^\lambda$, $\lambda \vdash k$, form a linear basis of $L(S_{2k},H_k)$
and satisfy the orthogonality relation
\begin{equation} \label{eq:ortho-omega}
\omega^\lambda \star \omega^\mu= \delta_{\lambda \mu} \frac{(2k)!}{2^k k!}
\frac{1}{f^{2\lambda}} \omega^\lambda
\end{equation}
(\cite[VII, (1.4)]{Mac}).

For a complex number $z$, 
we define the element $T^{\mathrm{O}}(\sigma;z)$ in $L(S_{2k},H_k)$ by
$$
T^{\mathrm{O}}(\sigma;z)= z^{\ell(\mu)} \qquad (\sigma \in S_{2k}),
$$
where $\mu$ is the coset-type of $\sigma$.
We emphasize that $T^{\mathrm{O}}$ is different from $T^{\mathrm{U}}$.
The function $T^{\mathrm{O}}(\cdot;z)$ is expanded in terms of $\omega^\lambda$
as follows (\cite[(4.5)]{CM}):
$$ 
T^{\mathrm{O}}(\cdot;z)= \frac{2^k k!}{(2k)!} \sum_{\lambda \vdash k }
f^{2\lambda} D_\lambda(z) \omega^\lambda,
$$
where $D_\lambda(z)$ is the polynomial in $z$ given by
$$
D_\lambda(z)= \prod_{(i,j) \in \lambda}
(z+2j-i-1).
$$

The {\it orthogonal Weingarten function} $\mathrm{Wg}^{\mathrm{O}}(\cdot;z)$
on $S_{2k}$ is defined by
$$
\mathrm{Wg}^{\mathrm{O}}(\cdot;z) = \frac{2^k k!}{(2k)!}
\sum_{\begin{subarray}{c} \lambda \vdash k \\ D_\lambda(z) \not=0 \end{subarray}}
\frac{f^{2\lambda}}{D_\lambda(z)} \omega^\lambda,
$$
which is the pseudo-inverse element of $T^{\mathrm{O}}(\cdot;z)$, i.e., the unique element in $L(S_k,H_k)$ satisfying
\begin{equation} \label{eq:TWT-ortho}
T^{\mathrm{O}}(\cdot;z)\star \mathrm{Wg}^{\mathrm{O}}(\cdot;z)\star T^{\mathrm{O}}(\cdot;z) =T^{\mathrm{O}}(\cdot;z).
\end{equation}
In particular, if $D_\lambda(z) \not=0$ for all partitions $\lambda$ of $k$,
functions $T^{\mathrm{O}}(\cdot;z)$ and $ \mathrm{Wg}^{\mathrm{O}}(\cdot;z)$
are their inverse of each other and satisfy 
$T^{\mathrm{O}}(\cdot;z)\star \mathrm{Wg}^{\mathrm{O}}(\cdot;z)=\bm{1}_k$.
Those relations follow from \eqref{eq:ortho-omega} and 
expansions of 
$T^{\mathrm{O}}(\cdot;z)$ and $ \mathrm{Wg}^{\mathrm{O}}(\cdot;z)$
in terms of $\omega^\lambda$.

\subsubsection{Integrals on orthogonal groups}

The (real) orthogonal group is 
$\mathrm{O}(n)=\{ R \in \mathrm{GL}(n,\mathbb{R}) \ | \ 
R \trans{R}=I_n\}$ and has the Haar probability measure.

\begin{thm}[Weingarten calculus for orthogonal groups \cite{CollinsSniady,CM}]
Let $R=(r_{ij})_{1 \le i,j \le n}$ be an $n \times n$ Haar orthogonal
matrix.
For two sequences $\bm{i}=(i_1,\dots,i_{2k})$ and
$\bm{j}=(j_1,\dots,j_{2k})$ of positive integers in $[n]$,
we have
$$
\bE [ r_{i_1 j_1} \cdots r_{i_{2k} j_{2k}}]=
\sum_{\sigma,\tau \in M_{2k}}
\Delta_\sigma(\bm{i}) \Delta_\tau(\bm{j}) \mathrm{Wg}^{\mathrm{O}}(\sigma^{-1}\tau;n).
$$
Here $\Delta_{\sigma}(\bm{i})$ is defined by
\begin{equation} \label{eq:Delta-ft}
\Delta_{\sigma}(\bm{i})= \prod_{s=1}^k \delta_{i_{\sigma(2s-1)}, i_{\sigma(2s)}}.
\end{equation}
Furthermore, $\bE [ r_{i_1 j_1} \cdots r_{i_{2k+1} j_{2k+1}}]=0$
for any $i_1,\dots,i_{2k+1},j_1,\dots,j_{2k+1}$.
\end{thm}

\begin{example} \label{example:WgO}
$\mathrm{Wg}^{\mathrm{O}}(\mathrm{id}_2;n)= \frac{1}{n}$;
$$
\mathrm{Wg}^{\mathrm{O}}(\mathrm{id}_4;n)= \frac{n+1}{n(n+2)(n-1)},
\qquad 
\mathrm{Wg}^{\mathrm{O}}(\sigma_{(2)};n)= \frac{-1}{n(n+2)(n-1)}.
$$
We can see more examples in \cite{CollinsSniady,CM}.
\end{example}

\subsection{Weingarten calculus for symplectic groups}

In this subsection, we give the Weingarten calculus for 
symplectic groups. 
It was given in \cite{CollinsStolz,Dah}, however,
unlike their descriptions, we employ 
the theory of twisted Gelfand pairs (\cite[VII, Examples 1-10, 1-11, 2-6, and 2-7]{Mac}).

\subsubsection{Twisted Gelfand pairs}

Let $\epsilon$ be the signature function on $S_{2k}$ and 
consider the linear space
$$
L^\epsilon(S_{2k},H_k)=\{f \in L(S_{2k}) \ | \ 
f(\zeta \sigma)= f(\sigma \zeta)= \epsilon(\zeta) f(\sigma)
\quad (\sigma \in S_{2k}, \ \zeta \in H_k)\}.
$$
The space $L^\epsilon(S_{2k},H_k)$ is closed under the 
convolution $*$,
and it becomes a $\bC$-algebra.
Furthermore, it is known that
$L^\epsilon(S_{2k},H_k)$ is a commutative algebra,
which means that the triple $(S_{2k},H_k,\epsilon|_{H_k})$ is, 
by definition,
a {\it twisted Gelfand pair}
(\cite[VII, Example 2-6]{Mac}).
It is immediate to see that,
if $f_1,f_2 \in L^\epsilon(S_{2k},H_k)$, then 
$f_1 \star f_2 = (2^k k!)^{-1} f_1 *f_2 \in L^\epsilon(S_{2k},H_k)$.
Thus, $L^\epsilon(S_{2k},H_k)$ is a commutative algebra under the product 
$\star$ with the identity element
$$
\bm{1}_k^\epsilon (\sigma)= \begin{cases}
\epsilon(\sigma) & \text{if $\sigma \in H_k$}, \\
0 & \text{if $\sigma \in S_{2k} \setminus H_k$}.
\end{cases}
$$

For each partition $\lambda =(\lambda_1,\lambda_2,\dots) \vdash k$, 
we define
the {\it twisted spherical function} $\pi^\lambda$ 
of the twisted Gelfand pair by
$$
\pi^\lambda= (2^k k!)^{-1} \chi^{\lambda \cup \lambda} * \bm{1}_k^\epsilon
$$
with $\lambda \cup \lambda=(\lambda_1,\lambda_1,\lambda_2,\lambda_2,\dots)$.
The map 
$$
L(S_{2k}, H_k) \to L^\epsilon(S_{2k},H_k): f \mapsto  f^\epsilon, \qquad
f^\epsilon(\sigma)=\epsilon(\sigma) f(\sigma) \ (\sigma \in S_{2k})
$$
defines a $\mathbb{C}$-algebra isomorphism.
The twisted spherical function $\pi^\lambda$ is the
image of $\omega^{\lambda'}$ with  
$\pi^\lambda(\sigma_\mu)=\omega^{\lambda'}(\sigma_\mu)$.
Here $\lambda'=(\lambda_1',\lambda_2',\dots)$ 
is the conjugate partition of $\lambda$,
which is characterized by 
$(i,j) \in \lambda \Leftrightarrow
(j,i) \in \lambda'$.
Hence, the $\{ \pi^\lambda \ | \ \lambda \vdash k\}$
form a linear basis of $L^\epsilon(S_{2k},H_k)$ and satisfy
the orthogonality relation
 $$
 \pi^\lambda \star \pi^\mu= \delta_{\lambda \mu}
\frac{(2k)!}{2^k k!} \frac{1}{f^{\lambda \cup \lambda}} \pi^\lambda.
$$

\subsubsection{Symplectic Weingarten functions}

Let $z$ be a complex number and 
consider the function $T^{\mathrm{Sp}}(\cdot;z)$ in $L^\epsilon(S_{2k},H_k)$ defined by
$$
T^{\mathrm{Sp}}(\sigma;z)= (-1)^{k} \epsilon(\sigma) (-2z)^{\ell(\mu)}\qquad
(\sigma \in S_{2k}),
$$
where $\mu$ is the coset-type of $\sigma$.
Let
$$
D'_\lambda(z)= \prod_{(i,j) \in \lambda} (2z-2i+j+1).
$$
The expansion of $T^{\mathrm{Sp}}(\sigma;z)$ in terms of 
linear basis $\pi^\lambda$ is given as follows.

\begin{lem}
\begin{equation} \label{eq:T-pi}
T^{\mathrm{Sp}}(\cdot;z)= \frac{2^k k!}{(2k)!} \sum_{\lambda \vdash k}
f^{\lambda \cup \lambda} D'_\lambda(z) \pi^\lambda.
\end{equation}
\end{lem}

\begin{proof}
We first suppose that $z=n$, a positive integer.
We employ symmetric function theory.
Recall power-sum symmetric functions $p_\mu$ and 
twisted zonal functions $Z_\lambda'$. 
We only use the following properties for them here (see \cite[VII, Example 2-7]{Mac}).
For partitions $\lambda,\mu$ of $k$,
\begin{align*}
&p_\mu(1^n) = n^{\ell(\mu)}, \qquad 
Z_\lambda'(1^n) = D'_\lambda(n), \\
& p_\mu= \frac{2^k k!}{(2k)!} (-1)^{k-\ell(\mu)} 2^{-\ell(\mu)}\sum_{\lambda \vdash k}
 f^{\lambda \cup \lambda} \pi^\lambda(\sigma_\mu) Z_\lambda'.
\end{align*}
Here $(1^n)=(1,1,\dots,1)$ with $n$ times.
Hence, 
if $\mu$ is the coset-type of $\sigma \in S_{2k}$, then
$$
T^{\mathrm{Sp}}(\sigma;n) 
= \epsilon(\sigma)  (-1)^{k-\ell(\mu)} 2^{\ell(\mu)} p_\mu(1^n)
=  \epsilon(\sigma) \frac{2^k k!}{(2k)!}
\sum_{\lambda \vdash k} f^{\lambda \cup \lambda}\pi^\lambda(\sigma_\mu) D'_\lambda(n),
$$
which implies the desired formula for $z=n$.

Since the both sides on \eqref{eq:T-pi} are polynomials in $z$,
the equalities at all positive integers $z=n$ implies 
the ones at all complex numbers $z$.
\end{proof}

The {\it symplectic Weingarten function} with parameter $z$ is the 
function in $L^\epsilon(S_{2k},H_k)$ defined by
\begin{equation} \label{def:WgSp}
\mathrm{Wg}^{\rmSp}(\cdot;z)= \frac{2^k k!}{(2k)!} 
\sum_{\begin{subarray}{c} \lambda \vdash k \\ D'_\lambda(z) \not=0 \end{subarray}}
\frac{f^{\lambda \cup \lambda}}{D'_\lambda(z)} \pi^\lambda.
\end{equation}
Note that $T^{\mathrm{Sp}}(\sigma;z)= (-1)^k \epsilon(\sigma) T^{\mathrm{O}}(\sigma;-2z)$
and $\mathrm{Wg}^{\mathrm{Sp}}(\sigma;z)= 
(-1)^k \epsilon(\sigma) \mathrm{Wg}^{\mathrm{O}}(\sigma;-2z)$.
Equation \eqref{eq:TWT-ortho} is equivalent to 
\begin{equation} \label{eq:TWT}
T^{\mathrm{Sp}}(\cdot;z)  \star \mathrm{Wg}^{\rmSp}(\cdot;z) \star T^{\mathrm{Sp}} (\cdot;z)=T^{\mathrm{Sp}} (\cdot;z).
\end{equation}

\subsubsection{Integrals on symplectic groups}

Consider the vector space $\bC^{2n}$ of column vectors with standard basis
$(e_1,e_2,\dots,e_{2n})$.
Define the skew-symmetric bilinear form $\langle \cdot, \cdot \rangle$ on $\bC^{2n}$
by 
\begin{equation} \label{eq:SymplecticForm}
\langle v,w \rangle= \trans{v} J w
\end{equation}
with
$$
J=J_n :=\begin{pmatrix} O_n & I_n \\ -I_n & O_n \end{pmatrix}.
$$
Let $(e_1^\vee,e_2^\vee,\dots,e_{2n}^\vee)$ be the dual basis of $(e_1,\dots,e_{2n})$
with respect to $\langle \cdot, \cdot \rangle$:
$$
\langle e_i^\vee, e_j \rangle = \delta_{ij} \qquad (i,j \in [2n]).
$$
More specifically,
$(e_1^\vee,\dots,e_n^\vee,e_{n+1}^\vee,\dots,e_{2n}^\vee)=
(-e_{n+1},\dots,-e_{2n},e_1,\dots,e_n)$.
For convenience, we use the following notation: 
$$
\langle i,j \rangle :=\langle e_i,e_j \rangle = 
\langle e_i^\vee,e_j^\vee \rangle = 
\begin{cases} 1 & 
\text{if $1 \le i\le n$ and $j=i+n$}, \\
-1 & \text{if $1 \le j\le n$ and $i=j+n$}, \\
0 & \text{otherwise}.
\end{cases}
$$
Note that
\begin{equation} \label{eq:eeee}
e_i= \trans{J}e_i^\vee =\sum_{p=1}^{2n} \langle i,p \rangle e_p^\vee, \qquad
e_i^\vee=J e_i=\sum_{p=1}^{2n} \langle p,i \rangle e_p
\qquad (i=1,2,\dots,2n)
\end{equation}
and 
$J= \( \langle i, j \rangle \)_{1 \le i,j \le 2n}$.

For a $2n \times 2n$ matrix $X$, 
we define
the {\it dual} matrix of $X$ by
$$
\dual{X} := J \trans{X} \trans{J}.
$$
Then $\langle v, Xw \rangle = \langle \dual{X}v,w \rangle$
for all $v, w \in \bC^{2n}$.

We realize the (unitary) symplectic group $\rmSp(2n)$ 
by
$$
\rmSp(2n)= \{ S=(s_{ij})_{1 \le i,j \le 2n} \in \rmU(2n) \ | \ S\dual{S} =I_{2n}  \}.
$$
It is equipped with the Haar probability measure.
The following theorem was first given in \cite{CollinsStolz} without 
the explicit expression \eqref{def:WgSp} for $\mathrm{Wg}^{\rmSp}$.

\begin{thm} \label{thm:WgSp}
Let $S=(s_{ij})_{1 \le i,j \le 2n}$ be a
Haar symplectic matrix.
For two sequences
$\bm{i}=(i_1,\dots,i_{2k})$ and $\bm{j}=(j_1,\dots, j_{2k})$
of positive integers in $[2n]$,
we have
$$
\bE [ s_{i_1 j_1} s_{i_2 j_2} \cdots s_{i_{2k},j_{2k}} ] 
= \sum_{\sigma,\tau \in M_{2k}} \Delta_\sigma'(\bm{i}) 
\Delta_{\tau}'(\bm{j}) \mathrm{Wg}^{\rmSp} (\sigma^{-1}\tau; n).
$$
Here 
the symbol  $\Delta_{\sigma}'(\bm{i}) \in \{0,1,-1\}$ is defined by
\begin{equation} \label{eq:Delta'-ft}
\Delta_{\sigma}'(\bm{i}) := \prod_{r=1}^k \langle i_{\sigma(2r-1)}, i_{\sigma(2r)}
\rangle.
\end{equation}
Furthermore, $\bE [ s_{i_1 j_1} \cdots s_{i_{2k+1} j_{2k+1}}]=0$
for any $i_1,\dots,i_{2k+1},j_1,\dots,j_{2k+1}$.
\end{thm}

We postpone the proof of this theorem in the next subsubsection.

\begin{example} \label{example:WgSp}
$\mathrm{Wg}^{\mathrm{Sp}}(\sigma;n)=  \frac{\epsilon(\sigma)}{2n}$
for $\sigma \in S_2$;
\begin{align*}
\mathrm{Wg}^{\mathrm{Sp}}(\sigma;n)=& \epsilon(\sigma) \frac{2n-1}{4n(n-1)(2n+1)} \qquad \text{for $\sigma \in H_2$}; \\
\mathrm{Wg}^{\mathrm{Sp}}(\sigma;n)=& \epsilon(\sigma) \frac{1}{4n(n-1)(2n+1)} \qquad \text{for $\sigma \in S_4 \setminus H_2$}.
\end{align*}
\end{example}

\subsubsection{Proof of Theorem \ref{thm:WgSp}}

This proof was given by Collins and Stolz \cite{CollinsStolz} 
(see also \cite{CM}),
however, they did not give any explicit expression for the symplectic Weingarten function.
We here reconstruct their proof and observe 
how $\mathrm{Wg}^{\rmSp}$ arises.

We first introduce the following useful notation.
For a permutation $\sigma \in S_{2k}$ and 
two matrices $X=(x_{ij})_{1 \le i,j \le n}$ and $Y=(y_{ij})_{1 \le i,j \le n}$,
put
$$
\mcal{T}_\sigma(X,Y)= \sum_{p_1,\dots,p_{2k} \in [n]}
\prod_{r=1}^k 
x_{p_{\sigma(2r-1)}, p_{\sigma(2r)}}
y_{p_{2r-1}, p_{2r}}
= \sum_{q_1,\dots,q_{2k} \in [n]} \prod_{r=1}^{k}
x_{q_{2r-1}, q_{2r}}
y_{q_{\sigma^{-1}(2r-1)}, q_{\sigma^{-1}(2r)}}.
$$
Note that $\mcal{T}_{\sigma^{-1}}(X,Y)=\mcal{T}_\sigma(Y,X)$
and $\mcal{T}_{\sigma_{\mu}} (X,Y)= \prod_{i=1}^{\ell(\mu)} 
\mcal{T}_{\sigma_{(\mu_i)}}(X,Y)$,
where $\sigma_\mu$ is defined in \eqref{eq:sigma-mu}.
Since 
$$
\mcal{T}_{\sigma_{(m)}}(X,Y)=
\sum_{p_1,\dots,p_{2m}}
x_{p_1,p_{2m}} y_{p_1, p_2} x_{p_2,p_3} y_{p_3, p_4} \cdots x_{p_{2m-2},p_{2m-1}}  
 \cdots y_{p_{2m-1},p_{2m}},
$$
we have 
\begin{equation} \label{eq:T-Tr-Sp}
\mcal{T}_{\sigma_{(m)}}(X,Y)
 = 
 \begin{cases} \Tr [(XY)^m] & \text{if $X$ is a symmetric matrix}, \\
  -\Tr [(XY)^m] & \text{if $X$ is a skew-symmetric matrix}. 
 \end{cases}
\end{equation}
As a particular case, we can see 
\begin{equation} \label{eq:TJJ}
\mcal{T}_\sigma(J,J)= T^{\mathrm{Sp}}(\sigma;n)
\qquad (\sigma \in S_{2k}).
\end{equation}
Indeed, the skew-symmetry of $J$ implies that
the function $\sigma \mapsto \mcal{T}_\sigma(J,J)$ 
belongs to $L^\epsilon(S_{2k},H_k)$,
and hence
it is enough to check  $\mcal{T}_{\sigma_{\mu}}(J,J)= 
T^{\mathrm{Sp}}(\sigma_\mu;n)$ for all partitions $\mu$.
However, it follows 
from \eqref{eq:T-Tr-Sp} that 
$\mcal{T}_{\sigma_\mu}(J,J)=\prod_{i=1}^{\ell(\mu)} \mcal{T}_{\sigma_{(\mu_i)}}(J,J)= 
\prod_{i=1}^{\ell(\mu)} [- \Tr (-I_{2n})^{\mu_i}] =(-1)^{k-\ell(\mu)} (2n)^{\ell(\mu)}
=T^{\mathrm{Sp}}(\sigma_{\mu};n)$.

We next recall the invariant theory for symplectic groups.
Consider the tensor product $(\bC^{2n})^{\otimes 2k}$
and define a bilinear form on $(\bC^{2n})^{\otimes 2k}$ by
$$
\left\langle \bigotimes_{j=1}^{2k} v_j, \bigotimes_{j=1}^{2k} w_j 
\right\rangle :=
\prod_{j=1}^{2k} \langle v_j, w_j \rangle,
\qquad (v_1,\dots,v_{2k}, w_1,\dots, w_{2k} \in \bC^{2n}),
$$
where the skew-symmetric bilinear form $\langle v,
w\rangle$ on the right hand side is defined in \eqref{eq:SymplecticForm}.
Note that this bilinear form on $(\bC^{2n})^{\otimes 2k}$ is symmetric.

Put 
$$
\theta_{k}= \sum_{p_1,\dots,p_k \in [2n]} e_{p_1}^\vee \otimes e_{p_1} \otimes 
\cdots \otimes e_{p_k}^\vee \otimes e_{p_k} \in (\bC^{2n})^{\otimes 2k}.
$$
The symmetric group $S_{2k}$ acts on $(\bC^{2n})^{\otimes 2k}$ by
$$
\rho_{2k}(\sigma) (v_1 \otimes \cdots \otimes v_{2k}) = 
v_{\sigma^{-1}(1)} \otimes \cdots \otimes v_{\sigma^{-1}(2k)},
$$
while the symplectic group $\rmSp(2n)$ acts by
$$
S (v_1 \otimes \cdots \otimes v_{2k}) = Sv_1 \otimes \cdots S v_{2k}.
$$
Then the First Fundamental Theorem for the symplectic group states that
$\{\rho_{2k}(\sigma) \theta_k \ | \ \sigma \in M_{2k}\}$ spans 
the vector subspace
of $(\bC^{2n})^{\otimes 2k}$ consisting of  
invariant elements under the action of $\rmSp(2n)$ (see, e.g., \cite[Theorem 5.3.4]{GW}).

Let us go back to the proof of Theorem \ref{thm:WgSp}.
Let $\mathsf{G}$ be the symmetric matrix 
\begin{equation} \label{eq:defG}
\mathsf{G}=(\langle \rho_{2k}(\sigma) \theta_k, \rho_{2k}(\tau) \theta_k\rangle)_{\sigma, \tau \in M_{2k}}
\end{equation} 
and $\mathsf{W}=(
\mathsf{w}(\sigma,\tau))_{\sigma,\tau \in M_{2k}}$  the pseudo-inverse matrix 
of $\mathsf{G}$,
i.e., $\mathsf{W}$ is the unique symmetric matrix satisfying $\mathsf{G W G}=
\mathsf{G}$.

Let $S=(s_{ij})_{1 \le i,j \le 2n}$ be a Haar symplectic matrix.
Since each matrix element $s_{ij}$ is expressed as 
$s_{ij} = \langle e_i^\vee, S e_j \rangle$,
we have
\begin{align*}
 \bE [ s_{i_1 j_1} \cdots s_{i_{2k} j_{2k}} ] 
=&  \bE \Big[ \langle 
 e_{i_1}^\vee \otimes \cdots  \otimes e_{i_{2k}}^\vee,
 S( e_{j_1} \otimes \cdots  \otimes e_{j_{2k}})
 \rangle \Big] \\
 =& \Big\langle 
 e_{i_1}^\vee \otimes \cdots  \otimes e_{i_{2k}}^\vee,
 \bE[S( e_{j_1} \otimes \cdots  \otimes e_{j_{2k}})]
 \Big\rangle.
\end{align*}
As mentioned above, for any $\bm{v} \in V^{\otimes 2k}$,
the $\rmSp(2n)$-invariant vector $\bE [S\bm{v}]$ 
can be expanded  in terms of $\rho_{2k}(\sigma) \theta_k$.
From a discussion  parallel to the proof of \cite[Theorem 2.1]{CM}
(see also \cite{CollinsStolz}), 
the expansion is given by
$$
\bE [S\bm{v}]= \sum_{\sigma \in M_{2k}} \(\sum_{\tau \in M_{2k}} \mathsf{w}(\sigma,\tau)
\langle \bm{v}, \rho_{2k}(\tau) \theta_k \rangle\) \rho_{2k}(\sigma) \theta_k.
$$
Applying this to the previous equation, we obtain
$$
 \bE [ s_{i_1 j_1} \cdots s_{i_{2k} j_{2k}} ] =  
\sum_{\sigma, \tau \in M_{2k}} \mathsf{w}(\sigma,\tau) \cdot
\langle e_{i_1}^\vee \otimes \cdots  \otimes e_{i_{2k}}^\vee, 
\rho_{2k}(\sigma) \theta_k \rangle 
\cdot \langle e_{j_1} \otimes \cdots  \otimes e_{j_{2k}}, \rho_{2k}(\tau)\theta_k \rangle.
$$
Here the values of bilinear forms are computed as follows:
\begin{align*}
&\langle e_{i_1}^\vee \otimes \cdots  \otimes e_{i_{2k}}^\vee, 
\rho_{2k}(\sigma) \theta_k \rangle 
= \langle \rho_{2k}(\sigma^{-1}) (e_{i_1}^\vee \otimes \cdots  \otimes e_{i_{2k}}^\vee),  \theta_k \rangle \\
=& \langle e_{i_{\sigma(1)}}^\vee \otimes \cdots  \otimes e_{i_{\sigma(2k)}}^\vee,  \theta_k \rangle 
= \prod_{r=1}^k \(\sum_{p=1}^{2n} \langle e_{i_{\sigma(2r-1)}}^\vee, e_{p}^\vee \rangle
\langle e_{i_{\sigma(2r)}}^\vee,e_p \rangle \) \\
=& \Delta_\sigma'(\bm{i}),
\end{align*}
and, similarly,  
$\langle e_{j_1} \otimes \cdots  \otimes e_{j_{2k}}, \rho_{2k}(\tau)\theta_k \rangle = \Delta_{\tau}'(\bm{j})$.
In conclusion, Theorem \ref{thm:WgSp} follows from the next lemma.

\begin{lem} \label{lem:WgSp}
For $\sigma,\tau \in M_{2k}$, 
$\mathsf{w}(\sigma,\tau)= \mathrm{Wg}^{\rmSp}(\sigma^{-1}\tau; n)$.
\end{lem}

\begin{proof}
If we put
$$
\widetilde{T}(\sigma)= \langle \theta_k, \rho_{2k}(\sigma)\theta_k \rangle \qquad (\sigma \in S_{2k}),
$$
then $\mathsf{G}= ( \widetilde{T}(\sigma^{-1}\tau))_{\sigma,\tau \in M_{2k}}$.
Using \eqref{eq:eeee} we have
\begin{align*}
\theta_k=& \sum_{p_1,\dots,p_{2k}} \langle p_1, p_2 \rangle \langle p_3, p_4 \rangle
\cdots \langle p_{2k-1}, p_{2k} \rangle e_{p_1}^\vee \otimes e_{p_2}^\vee \otimes \cdots 
\otimes e_{p_{2k}}^\vee \\
=& \sum_{q_1,\dots,q_{2k}} \langle q_1, q_2 \rangle \langle q_3, q_4 \rangle
\cdots \langle q_{2k-1}, q_{2k} \rangle e_{q_1} \otimes e_{q_2} \otimes \cdots 
\otimes e_{q_{2k}},
\end{align*}
and hence
$$
\tilde{T}(\sigma)= \sum_{q_1,\dots,q_{2k}} \langle q_{\sigma(1)},q_{\sigma(2)} \rangle
\cdots \langle q_{\sigma(2k-1)},q_{\sigma(2k)} \rangle
\langle q_{1},q_{2} \rangle \cdots \langle q_{2k-1},q_{2k} \rangle
=\mcal{T}_\sigma(J,J),
$$
which implies $\widetilde{T}(\sigma)= T^{\mathrm{Sp}}(\sigma;n)$
by  \eqref{eq:TJJ}.
The matrix $(\mathrm{Wg}^{\rmSp}(\sigma^{-1}\tau; n))_{\sigma,\tau \in M_{2k}}$
is therefore the pseudo-inverse matrix of $\mathsf{G}= ( T^{\mathrm{Sp}}(\sigma^{-1}\tau;n))_{\sigma,\tau \in M_{2k}}$ by 
\eqref{eq:TWT}.
\end{proof}

\section{Circular ensembles}

From now we consider random matrix ensembles $\mcal{S}$ associated with classical 
symmetric spaces $G/K$.
As we mentioned in Introduction, such ensembles are realized 
in the following way.
$$
G/K \simeq \mcal{S}; \qquad G \ni g \mapsto \Omega(g)^{-1} g \in \mcal{S}.
$$
Here $\Omega$ is an involution on $G$ and $K$ is the fixed-point set of $\Omega$.
If $X$ is a Haar random matrix picked up from $G$, then
the matrix $V:= \Omega(X)^{-1} X$ is a random matrix associated with 
$G/K$.
We consider 
the seven series of random matrices associated with compact symmetric spaces
given in Figure \ref{ListCSS}.

In this section, we deal with the most important classes:
circular orthogonal ensembles (COE) and circular symplectic ensembles (CSE).

\subsection{Class A I}

{\it The setting for A I}:
$G=\rmU(n)$, $K=\rmO(n)$, $\Omega(g)= \overline{g}$.
$\mcal{S}$ consists of $n \times n$ symmetric unitary matrices.

When $U$ is an $n \times n$ Haar  unitary matrix,
a random matrix corresponding to $\rmU(n)/ \rmO(n)$ is defined by
$$
V=V^{\mathrm{A \, I}}= \Omega(U)^{-1} U= \trans{U} U
$$
and is said to be a COE matrix.
The Weingarten calculus for a COE matrix is constructed in \cite{M-COE}.
For completeness of this paper, we review main results in \cite{M-COE}.
Applying the Weingarten calculus for $\mathrm{U}(n)$, we can obtain
the following theorem.

\begin{thm}[Theorem 1.1 and Proposition 3.1 in \cite{M-COE}]
Let $V=V^{\mathrm{A \, I}}=(v_{ij})_{1 \le i,j \le n}$ be an $n \times n$ COE matrix.
For two sequences $\bm{i}=(i_1,\dots,i_{2k})$ and 
$\bm{j}=(j_1,\dots,j_{2k})$, we have
$$
\bE[ v_{i_1 i_2} v_{i_3 i_4} \cdots v_{i_{2k-1} i_{2k}}
\overline{ v_{j_1 j_2} v_{j_3 j_4} \cdots v_{j_{2k-1} j_{2k}}}]
= \sum_{\sigma \in S_{2k}} \delta_\sigma(\bm{i},\bm{j}) 
\mathrm{Wg}^{\mathrm{A \, I}}(\sigma;n)
$$
with  the convolution $\mathrm{Wg}^{\mathrm{A \, I}}(\cdot;n):
=T^{\mathrm{O}}(\cdot; n) * \mathrm{Wg}^{\mathrm{U}}(\cdot;n)$
in $L(S_{2k})$.
If $k \not=l$ then
$$
\bE[ v_{i_1 i_2} v_{i_3 i_4} \cdots v_{i_{2k-1} i_{2k}}
\overline{ v_{j_1 j_2} v_{j_3 j_4} \cdots v_{j_{2l-1} j_{2l}}}]
$$
always vanishes.
Moreover, the function $\mathrm{Wg}^{\mathrm{A \, I}}(\cdot;n)$
belongs to $L(S_{2k},H_k)$ and 
coincides with 
the orthogonal Weingarten function with parameter $n+1$, i.e., 
$$
\mathrm{Wg}^{\mathrm{A \, I}}(\sigma;n)=
\mathrm{Wg}^{\mathrm{O}}(\cdot;n+1).
$$
\end{thm}

As a corollary of this theorem, the following moments for a
single entry are computed (see \cite[Theorems 4.1 and 4.2]{M-COE}): 
\begin{align*}
\bE[ |v_{ii}|^{2k}]=& \frac{2^k k!}{(n+1)(n+3) \cdots (n+2k-1)},\\
\bE[ |v_{ij}|^{2k}]=& \frac{k!}{n(n+1)(n+2) \cdots (n+k-2)(n+2k-1)}
\quad (i \not=j).
\end{align*}
The first equation can be obtained easily from the previous theorem,
but the derivation of the second one is somewhat complicated.

\begin{example} 
From Example \ref{example:WgO} we have 
$\mathrm{Wg}^{\mathrm{A\, I}}(\sigma;n)=\mathrm{Wg}^{\mathrm{O}}(\sigma;n+1)= \frac{1}{n+1}$ for $\sigma \in S_2$;
$$
\mathrm{Wg}^{\mathrm{A\, I}}(\sigma;n) =\mathrm{Wg}^{\mathrm{O}}(\sigma;n+1)= 
\begin{cases}
\frac{n+2}{n(n+1)(n+3)} & \text{for $\sigma \in H_2$} \\
\frac{-1}{n(n+1)(n+3)} & \text{for $\sigma \in S_4 \setminus H_2$}. 
\end{cases}
$$
\end{example}

\subsection{Class A II}

{\it The setting for A II}:
$G= \rmU(2n)$, $K=\rmSp(2n)$, $\Omega(g)= (\dual{g})^{-1}$.
$\mcal{S}$ consists of $2n \times 2n$ unitary matrices $g$
satisfying $\dual{g}=g$.

When $U$ is a $2n \times 2n$ Haar unitary matrix,
a random matrix corresponding to $\rmU(2n)/\rmSp(2n)$ is defined by
$$
V=V^{\mathrm{A \, II}}= \dual{U} U
$$
and is said to be a CSE matrix.

We would like to compute  mixed moments 
for $v_{ij}$ and $\overline{v_{ij}}$.
In order to simplify the notation, we deal with
$$
\Tilde{v}_{ij}:=\langle e_i, V e_j \rangle \qquad (1 \le i,j \le 2n)
$$
instead of $v_{ij}=\langle e_i^\vee, V e_j \rangle$.
More specifically,
$$
v_{ij}= \begin{cases}
-\Tilde{v}_{i+n,j} & \text{if $1 \le i \le n$},\\
\Tilde{v}_{i-n,j} & \text{if $n+1 \le i \le 2n$}.
\end{cases}
$$

\begin{thm} \label{thm:A II}
Let $V$ be a $2n \times 2n$ CSE matrix.
For two sequences $\bm{i}=(i_1,\dots,i_{2k})$ and $\bm{j}=(j_1,\dots,j_{2k})$
in $[2n]^{\times 2k}$, we have
$$
\bE [\Tilde{v}_{i_1,i_2} \Tilde{v}_{i_3, i_4} \cdots \Tilde{v}_{i_{2k-1}, i_{2k}} 
\overline{ \Tilde{v}_{j_1,j_2} \Tilde{v}_{j_3, j_4} \cdots 
\Tilde{v}_{j_{2k-1}, j_{2k}} }]  
= \sum_{\sigma \in S_{2k}} \delta_{\sigma}(\bm{i},\bm{j}) 
\mathrm{Wg}^{\mathrm{A \, II}}(\sigma;n)
$$
with $\mathrm{Wg}^{\mathrm{A \, II}}(\cdot;n):=
T^{\mathrm{Sp}}(\cdot;n) * \mathrm{Wg}^{\mathrm{U}}(\cdot;2n)$.
If $k \not=l$ then
$$
\bE [\Tilde{v}_{i_1,i_2} \Tilde{v}_{i_3, i_4} \cdots \Tilde{v}_{i_{2k-1}, i_{2k}} 
\overline{ \Tilde{v}_{j_1,j_2} \Tilde{v}_{j_3, j_4} \cdots 
\Tilde{v}_{j_{2l-1}, j_{2l}} }]  
 =0.
$$
\end{thm}

\begin{proof}
Each element of a $2n \times 2n$ CSE matrix $V =\dual{U}U$ is given as 
$$
\Tilde{v}_{i,i'}= \langle e_i, \dual{U} U e_{i'} \rangle = \langle U e_i, Ue_{i'} \rangle
= \sum_{p,q \in [2n]} u_{pi} u_{qi'} \langle p,q \rangle
\qquad (i,i' \in [2n]).
$$
Hence 
\begin{align*}
&\bE [\Tilde{v}_{i_1,i_2} \Tilde{v}_{i_3, i_4} \cdots \Tilde{v}_{i_{2k-1}, i_{2k}} 
\overline{ \Tilde{v}_{j_1,j_2} \Tilde{v}_{j_3, j_4} \cdots 
\Tilde{v}_{j_{2k-1}, j_{2k}} }]  \\
=& \sum_{\bm{p}=(p_1,\dots,p_{2k})} \sum_{\bm{q}=(q_1,\dots,q_{2k})}
\bE[ u_{p_1 i_1} \cdots u_{p_{k} i_{k}}
\overline{
u_{q_1 j_1} \cdots u_{q_{k}j_{k}}}] 
\prod_{r=1}^k \langle p_{2r-1},p_{2r} \rangle 
\langle q_{2r-1}, q_{2r} \rangle.
\end{align*}
(When $k \not=l$, 
$\bE[ u_{p_1 i_1} \cdots u_{p_{k} i_{k}}
\overline{
u_{q_1 j_1} \cdots u_{q_{l}j_{l}}}] \equiv0$.)
Applying the Weingarten calculus for a Haar unitary matrix $U$,
we have
$$
\bE [\Tilde{v}_{i_1,i_2} \Tilde{v}_{i_3, i_4} \cdots \Tilde{v}_{i_{2k-1}, i_{2k}} 
\overline{ \Tilde{v}_{j_1,j_2} \Tilde{v}_{j_3, j_4} \cdots 
\Tilde{v}_{j_{2k-1}, j_{2k}} }]  
= \sum_{\sigma \in S_{2k}} \delta_\sigma(\bm{i},\bm{j}) 
\sum_{\tau \in S_{2k}} \mathrm{Wg}^{\mathrm{U}}(\tau^{-1} \sigma;2n) 
\widetilde{T}(\tau),
$$
where
\begin{align*}
\widetilde{T}(\tau)=& \sum_{\bm{p}=(p_1,\dots,p_{2k}) \in [2n]^{\times 2k}} \sum_{\bm{q}=(q_1,\dots,q_{2k}) \in [2n]^{\times 2k}}
\delta_{\tau}(\bm{p},\bm{q}) \prod_{r=1}^k \langle p_{2r-1},p_{2r} \rangle 
\langle q_{2r-1}, q_{2r} \rangle \\
=& \mcal{T}_\tau(J,J)
=T^{\mathrm{Sp}}(\tau; n).
\end{align*}
The last equality above follows from \eqref{eq:TJJ}.
\end{proof}

\begin{prop} 
The function $\mathrm{Wg}^{\mathrm{A \, II}}(\cdot;n)
=T^{\mathrm{Sp}}(\cdot;n) * \mathrm{Wg}^{\mathrm{U}}(\cdot;2n)$ coincides with
the symplectic Weingarten function with parameter $n-\tfrac{1}{2}$.
Specifically,
for each $\sigma \in S_{2k}$, we have
$$ 
\mathrm{Wg}^{\mathrm{A \, II}}(\sigma;n)=
\mathrm{Wg}^{\rmSp}
(\sigma;n-\tfrac{1}{2}).
$$
\end{prop}

\begin{proof}
It follows from \eqref{eq:T-pi} and \eqref{def-uWg} that
\begin{align*}
T^{\mathrm{Sp}}(\cdot;n)* \mathrm{Wg}^{\mathrm{U}}( \cdot;2n)
=& \frac{2^k k!}{(2k)!} \sum_{\lambda \vdash k}
f^{\lambda \cup \lambda} D_\lambda'(n) \frac{1}{(2k)!}
\sum_{\mu \vdash 2k} \frac{f^\mu}{C_\mu(2n)} \pi^\lambda * \chi^\mu. 
\end{align*}
Since $\pi^\lambda * \chi^\mu =
\delta_{\lambda \cup\lambda,\mu}\frac{(2k)!}{f^{\lambda \cup\lambda}}
\pi^\lambda$ and
since 
$$
\frac{D_\lambda'(z)}{C_{\lambda \cup \lambda}(2z)}
= \frac{\prod_{(i,j) \in \lambda} (2z-2i+j+1)}{
\prod_{(i,j) \in \lambda} (2z+j-(2i-1))(2z+j-2i)}=
\frac{1}{\prod_{(i,j) \in \lambda} (2z+j-2i)} = \frac{1}{D_{\lambda}'(z-\tfrac{1}{2})},
$$
we have
$T^{\mathrm{Sp}} (\cdot;n)* \mathrm{Wg}^{\mathrm{U}}( \cdot;2n) 
= \frac{2^k k!}{(2k)!} \sum_{\lambda \vdash k}
\frac{f^{\lambda \cup \lambda}}{D'_\lambda(n-\frac{1}{2})} \pi^\lambda=
\mathrm{Wg}^{\rmSp}(\cdot; n-\tfrac{1}{2})$.
\end{proof}

\begin{example} 
From Example \ref{example:WgSp} we have
$\mathrm{Wg}^{\mathrm{A \, II}}(\sigma;n)
=\mathrm{Wg}^{\mathrm{Sp}}(\sigma;n-\tfrac{1}{2})=  
\frac{\epsilon(\sigma)}{2n-1}$
for $\sigma \in S_2$;
$$
\mathrm{Wg}^{\mathrm{A \, II}}(\sigma;n)=
\mathrm{Wg}^{\mathrm{Sp}}(\sigma;n-\tfrac{1}{2})=
\begin{cases}
 \epsilon(\sigma) \frac{n-1}{n(2n-1)(2n-3)} &
 \text{for $\sigma \in H_2$}, \\
\epsilon(\sigma) \frac{1}{2n(2n-1)(2n-3)}&
\text{for $\sigma \in S_4 \setminus H_2$}.
\end{cases}
$$
\end{example}

\section{Chiral ensembles}

In this section, we deal with random matrix ensembles associated 
with classes A~III, BD~I, C~II. 
They are known as chiral ensembles.

\subsection{Class A III}

{\it The setting for A  III}: 
$G=\rmU(n)$, $K=\rmU(a) \times \rmU(b)$, $\Omega(g)= I'_{ab} g I'_{ab}$,
where
$n=a+b$ with $a \ge b \ge 1$ and 
$$
I'_{ab}= \begin{pmatrix} I_a & O \\ O & -I_{b} \end{pmatrix}.
$$

Let $U$ be an $n \times n$ Haar unitary matrix.
A random matrix corresponding to
$\mathrm{U}(n)/\mathrm{U}(a) \times \mathrm{U}(b)=
\mathrm{SU}(n)/ \mathrm{S}(\mathrm{U}(a) \times \mathrm{U}(b))$
is defined by
$V= V^{\mathrm{A \, III}} = I'_{ab} U^* I'_{ab} U$.
For the sake of ease, we consider 
a Hermitian and unitary random matrix
$$
W= W^{\mathrm{ A \, III}} = U^* I'_{ab} U,
$$
instead of $V =I'_{ab}W$.

We define 
the function $T^{\mathrm{A\, III}}_{ab}$ in $\mcal{Z}(L(S_k))$
by
$$
T^{\mathrm{A \, III}}_{ab}(\sigma) = \Tr_\sigma(I'_{ab})
$$
where $\Tr_\sigma (A)=\prod_{j=1}^{\ell(\mu)} \Tr (A^{\mu_j})$
if $\mu$ is the cycle-type of $\sigma$.
More specifically, we have 
$$
T^{\mathrm{A \, III}}_{ab}(\sigma)= (a+b)^{\ell^{\mathrm{e}}(\mu)}
(a-b)^{\ell^{\mathrm{o}}(\mu)}.
$$
Here $\ell^{\mathrm{e}}(\mu)$
(resp. $\ell^{\mathrm{o}}(\mu)$)
is the number of parts $\mu_j$ with even lengths (resp. odd lengths).

\begin{thm} \label{thm:AIII}
Let $W=W^{\mathrm{A \, III}}$ be the random matrix defined as above.
For two sequences $\bm{i}=(i_1,\dots,i_{k})$
and $\bm{j}=(j_1,\dots,j_k)$, we have
$$
\bE[ w_{i_1 j_1} w_{i_2 j_2} \cdots w_{i_{k} j_{k}}]
= \sum_{\sigma \in S_{k}} \delta_{\sigma}(\bm{i},\bm{j}) 
\mathrm{Wg}^{\mathrm{A \, III}} (\sigma;a,b).
$$
Here the function $\mathrm{Wg}^{\mathrm{A \, III}} (\cdot;a,b)$
in $\mcal{Z}(L(S_{k}))$ is defined by
$$
\mathrm{Wg}^{\mathrm{A \, III}} (\cdot;a,b)
= T^{\mathrm{A \, III}}_{ab} * \mathrm{Wg}^{\mathrm{U}}(\cdot; n).
$$
\end{thm}

\begin{proof}
The random matrix $W$ is Hermitian and 
the distribution of $W$ is 
invariant under the unitary transform $W \mapsto g W g^*$,
where $g$ is a fixed unitary matrix.
Hence we can apply Theorem 3.1 in \cite{CMS}
and obtain
$$
\bE[ w_{i_1 j_1} w_{i_2 j_2} \cdots w_{i_{k} j_{k}}]
= \sum_{\sigma,\tau \in S_{k}} \delta_{\sigma}(\bm{i},\bm{j}) 
\mathrm{Wg}^{\mathrm{U}}(\tau^{-1} \sigma;n) \bE[ \Tr_\tau(W)].
$$
Here, wee see that $\bE[ \Tr_\tau(W)]= \bE[\Tr_\tau(I'_{ab})]
= \Tr_\tau(I'_{ab}) = T^{\mathrm{A \, III}}_{ab}(\tau)$,
and we obtain the desired identity.
\end{proof}

\begin{example} $n=a+b$.
$\mathrm{Wg}^{\mathrm{A \, III}}(\mathrm{id}_1;a,b)=\frac{a-b}{n}$; 
$$
\mathrm{Wg}^{\mathrm{A \, III}}(\mathrm{id}_2;a,b)=\frac{(a-b+1)(a-b-1)}{(n+1)(n-1)},\qquad 
\mathrm{Wg}^{\mathrm{A \, III}}((1 \ 2);a,b)=\frac{4ab}{n(n-1)(n+1)}.
$$
\end{example}

Recall power symmetric functions $p_\mu$ and Schur functions $s_\lambda$.
They have the relation
$$
p_\mu= \sum_{\lambda \vdash k} \chi^\lambda (\sigma) s_\lambda 
$$
if $\mu \vdash k$ is the cycle-type of $\sigma$ (\cite[I, (7.8)]{Mac}).
Furthermore, it is easy to see that
$$
p_\mu(1^a, (-1)^b)=
p_\mu(\underbrace{1,1,\dots,1}_a, \underbrace{-1,-1,\dots,-1}_b)= T^{\mathrm{A \, III}}_{ab}(\sigma).
$$
Hence we obtain the expansion of $T^{\mathrm{A \, III}}_{ab}$ in terms of
irreducible characters $\chi^\lambda$:
$$
T^{\mathrm{A \, III}}_{ab}= \sum_{\lambda \vdash k} s_{\lambda}
(1^a, (-1)^b) \chi^\lambda.
$$
On the other hand,
from
the well-known identity 
$s_\lambda(1^n)= \frac{f^\lambda C_\lambda(n)}{k!}$,
the unitary Weingarten function is expressed as
$$
\mathrm{Wg}^{\mathrm{U}}(\cdot ;n)= \frac{1}{(k!)^2}
\sum_{\lambda \vdash k} \frac{(f^\lambda)^2}{s_\lambda(1^n)} \chi^\lambda.
$$
Consequently,  using
the relation $\chi^\lambda *\chi^\mu = \frac{k!}{f^\lambda} \delta_{\lambda \mu} \chi^\lambda$,
we obtain
the following expansion of $\mathrm{Wg}^{\mathrm{A \, III}}(\cdot;a,b)$
in terms of $\chi^\lambda$:
$$
\mathrm{Wg}^{\mathrm{A \, III}}(\cdot;a,b)=
T^{\mathrm{A \, III}}_{ab} * \mathrm{Wg}^{\mathrm{U}}(\cdot; n)=
\frac{1}{k!} \sum_{\lambda \vdash k} f^\lambda \frac{
s_{\lambda}
(1^a,(-1)^b)}
{s_{\lambda}
(1^{a+b})} \chi^\lambda.
$$

\subsection{Class BD I}

{\it The setting for BD  I}: 
$G=\rmO(n)$, $K=\rmO(a) \times \rmO(b)$, $\Omega(g)= I'_{ab} g I'_{ab}$,
where
$n=a+b$ with $a \ge b \ge 1$.

The discussion is parallel to that in the previous subsection.
We deal with the symmetric and orthogonal random matrix
$$
W=W^{\mathrm{BD\,I}} = \trans{R} I'_{ab} R,
$$
where $R$ is an $n \times n$ Haar orthogonal matrix.
This random matrix is associated with
the symmetric space
$\mathrm{O}(n)/\mathrm{O}(a) \times \mathrm{O}(b)=
\mathrm{SO}(n)/ \mathrm{S}(\mathrm{O}(a) \times \mathrm{O}(b))$.

We define 
the function $T^{\mathrm{BD\, I}}_{ab}$ in $L(S_{2k},H_k)$ by
$$
T^{\mathrm{BD \, I}}(\sigma) = \mcal{T}_\sigma(I'_{ab},I_n).
$$
More specifically,
if $\sigma \in S_{2k}$ has the coset-type $\mu=(\mu_1,\mu_2,\dots,\mu_l)$, then
$$
T^{\mathrm{BD \, I}}_{ab}(\sigma)
=\prod_{i=1}^{\ell(\mu)} \Tr (I'_{ab})^{\mu_i}=
 (a+b)^{\ell^{\mathrm{e}}(\mu)}
(a-b)^{\ell^{\mathrm{o}}(\mu)}.
$$

\begin{thm}
Let $W=W^{\mathrm{BD \, I}}$ be the random matrix defined as above.
For a sequence $\bm{i}=(i_1,\dots,i_{2k})$, we have
$$
\bE[ w_{i_1 i_2} w_{i_3 i_4} \cdots w_{i_{2k-1} i_{2k}}]
= \sum_{\sigma \in M_{2k}} \Delta_{\sigma}(\bm{i}) 
\mathrm{Wg}^{\mathrm{BD \, I}} (\sigma;a,b).
$$
Here the function $\mathrm{Wg}^{\mathrm{BD \, I}} (\cdot;a,b)$
in $L(S_{2k},H_k)$ is defined by
$$
\mathrm{Wg}^{\mathrm{BD \, I}} (\cdot;a,b)
= T^{\mathrm{BD \, I}}_{ab} \star \mathrm{Wg}^{\mathrm{O}}(\cdot; n).
$$
\end{thm}

\begin{proof}
The proof is the same with that of Theorem \ref{thm:AIII}.
Apply Theorem 3.3 in \cite{CMS}.
\end{proof}

\begin{example} $n=a+b$.
$\mathrm{Wg}^{\mathrm{BD \, I}}(\mathrm{id}_2;a,b)=\frac{a-b}{n}$; 
\begin{align*}
\mathrm{Wg}^{\mathrm{BD \, I}}(\(\begin{smallmatrix} 1 & 2 & 3 & 4 \\
1 & 2 & 3 & 4 \end{smallmatrix}\);a,b)=&
\frac{(a-b)^2(n+1)-2n}{n(n+2)(n-1)},  \\
\mathrm{Wg}^{\mathrm{BD \, I}}(\(\begin{smallmatrix} 1 & 2 & 3 & 4 \\
1 & 3 & 2 & 4 \end{smallmatrix}\);a,b)=
\mathrm{Wg}^{\mathrm{BD \, I}}(\(\begin{smallmatrix} 1 & 2 & 3 & 4 \\
1 & 4 & 2 & 3 \end{smallmatrix}\);a,b)=&
\frac{4ab}{n(n+2)(n-1)}.
\end{align*}
\end{example}

Using zonal polynomials $Z_\lambda$ (see \cite[VII-2]{Mac}),
we can obtain the expansion
of 
$\mathrm{Wg}^{\mathrm{BD \, I}} (\cdot;a,b)$ in terms of zonal spherical functions
$\omega^\lambda$:
$$
\mathrm{Wg}^{\mathrm{BD \, I}}(\cdot;a,b)=
\frac{2^k k!}{(2k)!} \sum_{\lambda \vdash k} f^{2\lambda} \frac{
Z_{\lambda}
(1^a, (-1)^b)}
{Z_{\lambda}
(1^{a+b})} \omega^\lambda.
$$

\subsection{Class C II}

{\it The setting for C  II}: 
$G=\rmSp(2n)$, $K=\rmSp(2a) \times \rmSp(2b)$, $\Omega(g)= I''_{ab} g I''_{ab}$,
where
$n=a+b$ with $a \ge b \ge 1$ and 
$$
I''_{ab}= \begin{pmatrix} I'_{ab} & O \\ O & I'_{ab} \end{pmatrix}.
$$

The discussion is parallel to the previous subsection again.
We deal with the  random matrix
$$
W=W^{\mathrm{C\, II}} = \dual{S} I''_{ab} S.
$$
where $S$ is a $2n \times 2n$ Haar symplectic matrix.
It is associated with
the symmetric space
$\mathrm{Sp}(2n)/\mathrm{Sp}(2a) \times \mathrm{Sp}(2b)$.

We define 
the function $T^{\mathrm{C\, II}}_{ab}$ in $L^\epsilon(S_{2k},H_k)$
by
$$
T^{\mathrm{C \, II}}_{ab}(\sigma)=\mcal{T}_\sigma(J, JI_{ab}'').
$$
If $\sigma \in S_{2k}$ has the coset-type $\mu$, then
\begin{align*}
T^{\mathrm{C \, II}}_{ab}(\sigma)=& \epsilon(\sigma) 
\prod_{i=1}^{\ell(\mu)} T^{\mathrm{C \, II}}_{ab}(\sigma_{(\mu_i)}) 
=
\epsilon(\sigma)\prod_{i=1}^{\ell(\mu)} [-\Tr (J JI''_{ab} )^{\mu_i}] \\ 
=& (-1)^{k-\ell(\mu)} \epsilon(\sigma)\prod_{i=1}^{\ell(\mu)} \Tr (I''_{ab})^{\mu_i} 
=(-1)^{k-\ell(\mu)} \epsilon(\sigma) 2^{\ell(\mu)}
(a+b)^{\ell^{\mathrm{e}}(\mu)}
(a-b)^{\ell^{\mathrm{o}}(\mu)}.
\end{align*}
Here the second equality above follows by \eqref{eq:T-Tr-Sp}.

As in the case of class A II, we consider
$\Tilde{w}_{ij}=\langle e_i, W e_j \rangle$ instead of matrix elements $w_{ij}$.

\begin{thm} \label{thm:CII}
Let $W=W^{\mathrm{C \, II}}$ be the random matrix defined as above.
For a sequence $\bm{i}=(i_1,\dots,i_{2k})$, we have
$$
\bE[ \Tilde{w}_{i_1, i_2} \Tilde{w}_{i_3, i_4} \cdots \Tilde{w}_{i_{2k-1}, i_{2k}}]
= \sum_{\sigma \in M_{2k}} \Delta'_{\sigma}(\bm{i}) 
\mathrm{Wg}^{\mathrm{C \, II}} (\sigma;a,b).
$$
Here the function $\mathrm{Wg}^{\mathrm{C \, II}} (\cdot;a,b)$
in $L^\epsilon(S_{2k},H_k)$ is defined by
$$
\mathrm{Wg}^{\mathrm{C \, II}} (\cdot;a,b)
= T^{\mathrm{C \, II}}_{ab} \star \mathrm{Wg}^{\mathrm{Sp}}(\cdot; n).
$$
\end{thm}

\begin{proof}
Since 
$$
\Tilde{w}_{i,j}= \langle e_i, \dual{S} I''_{ab} S e_j \rangle = 
\langle S e_i,  I''_{ab} S e_j \rangle =
\sum_{p,q=1}^{2n} s_{pi} s_{qj} \langle e_p, I''_{ab} e_q \rangle,
$$
we have
$$
\bE[ \Tilde{w}_{i_1, i_2} \Tilde{w}_{i_3, i_4} \cdots \Tilde{w}_{i_{2k-1}, i_{2k}}]
= \sum_{p_1,\dots,p_{2k}} \bE[ s_{p_1 i_1} s_{p_2 i_2} \cdots s_{p_{2k} i_{2k}}]
\prod_{r=1}^k \langle p_{2r-1}, p_{2r} \rangle'',
$$
with $\langle p,q \rangle'':=\langle e_p,I_{ab}'' e_q \rangle$. 
The Weingarten calculus for symplectic groups gives
$$
\bE[ \Tilde{w}_{i_1, i_2} \Tilde{w}_{i_3, i_4} \cdots \Tilde{w}_{i_{2k-1}, i_{2k}}]=
\sum_{\sigma \in M_{2k}} \Delta'_\sigma(\bm{i}) 
\sum_{\tau \in M_{2k}} \mathrm{Wg}^{\mathrm{Sp}}(\tau^{-1} \sigma;n)
\widetilde{T}(\tau),
$$
where
\begin{align*}
\widetilde{T}(\tau) :=& \sum_{\bm{p}=(p_1,\dots,p_{2k})}
\Delta'_\tau(\bm{p}) \prod_{r=1}^k \langle p_{2r-1}, p_{2r} \rangle''
= \sum_{p_1,\dots,p_{2k}} \prod_{r=1}^k
\langle p_{\tau(2r-1)}, p_{\tau(2r)} \rangle \langle p_{2r-1}, p_{2r} \rangle'' \\
=&
\mcal{T}_\tau(J,JI_{ab}'')
=T^{\mathrm{C \, II}}_{ab}(\tau).
\end{align*}
\end{proof}

\begin{example} $n=a+b$.
$\mathrm{Wg}^{\mathrm{C \, II}}(\mathrm{id}_2;a,b)=\frac{a-b}{n}$; 
\begin{align*}
\mathrm{Wg}^{\mathrm{C \, II}}(\(\begin{smallmatrix} 1 & 2 & 3 & 4 \\
1 & 2 & 3 & 4 \end{smallmatrix}\);a,b)=&
\frac{(a-b)^2(2n-1)-n}{n(n-1)(2n+1)},  \\
(-1)\mathrm{Wg}^{\mathrm{C \, II}}(\(\begin{smallmatrix} 1 & 2 & 3 & 4 \\
1 & 3 & 2 & 4 \end{smallmatrix}\);a,b)=
\mathrm{Wg}^{\mathrm{C \, II}}(\(\begin{smallmatrix} 1 & 2 & 3 & 4 \\
1 & 4 & 2 & 3 \end{smallmatrix}\);a,b)=&
\frac{4ab}{n(n-1)(2n+1)}.
\end{align*}
\end{example}

Using twisted zonal polynomials $Z'_\lambda$ (see \cite[VII, Example 2-7]{Mac}),
we obtain the expansion of 
$\mathrm{Wg}^{\mathrm{C \, II}} (\cdot;a,b)$ in twisted spherical functions
$\pi^\lambda$:
$$
\mathrm{Wg}^{\mathrm{C \, II}}(\cdot;a,b)=
\frac{2^k k!}{(2k)!} \sum_{\lambda \vdash k} f^{\lambda \cup \lambda} \frac{
Z'_{\lambda}
(1^a, (-1)^b)}
{Z'_{\lambda}
(1^{a+b})} \pi^\lambda.
$$

\section{BdG ensembles}

In this section, we deal with matrix ensembles of types D~III and 
C~I.
They are called Bogoliubov-de Gennes (BdG) ensembles.

\subsection{Class  D III}

{\it The setting for D  III}: 
$G=\rmO(2n)$, $K=\mathrm{O}(2n) \cap \mathrm{Sp}(2n) \simeq \rmU(n)$, 
$\Omega(g)= (\dual{g})^{-1}$.

Let  $R$ be a $2n \times 2n$ Haar orthogonal matrix.
We consider the random matrix 
$$
V= V^{\mathrm{D \, III}}  = \dual{R} R,
$$
associated to the symmetric space $\mathrm{O}(2n)/\mathrm{U}(n)$.

We define the function $T^{\mathrm{D \, III}}_n$ on $S_{2k}$ by
$$
T^{\mathrm{D \, III}}_n(\sigma)= \mcal{T}_\sigma(I_{2n},J_n).
$$
It satisfies 
$T^{\mathrm{D \, III}}_n(\zeta \sigma \zeta')= \epsilon(\zeta) 
T^{\mathrm{D \, III}}_n(\sigma)$
for any $\sigma \in S_{2k}$ and $\zeta,\zeta' \in H_k$.
Note that
$$
T^{\mathrm{D \, III}}_n (\sigma_\mu)= 
\prod_{i=1}^{\ell(\mu)} T^{\mathrm{D \, III}}_n(\sigma_{(\mu_i)})
= \prod_{i=1}^{\ell(\mu)} \Tr (J^{\mu_i}) 
=
\begin{cases}
(-2n)^{\ell(\mu)} & \text{if $\mu$ is even},\\
0 & \text{otherwise}.
\end{cases}
$$
Here a partition $\mu$ is said to be {\it even}
if $\mu=2 \nu$ for some partition $\nu$.
In particular, $T^{\mathrm{D \, III}}_n(\sigma)=0$ ($\sigma \in S_{2k}$)
if $k$ is odd.

\begin{thm}
Let $V= V^{\mathrm{D \, III}}$ be the random matrix defined as above.
For a sequence $\bm{i}=(i_1,i_2,\dots,i_{2k})$ with an even number $k$, we have
$$
\bE [\Tilde{v}_{i_1,i_2} \Tilde{v}_{i_3,i_4} \cdots 
\Tilde{v}_{i_{2k-1}, i_{2k}}] = 
\sum_{\sigma \in M_{2k}} \Delta_\sigma(\bm{i}) \mathrm{Wg}^{\mathrm{D \, III}}(\sigma;n).
$$
Here $\mathrm{Wg}^{\mathrm{D \, III}}(\cdot;n)$ is the function on $S_{2k}$
defined by
$$
\mathrm{Wg}^{\mathrm{D \, III}}(\cdot;n) =T^{\mathrm{D \, III}}_n
 \star \mathrm{Wg}^{\mathrm{O}}(\cdot;2n).
$$ 
If $k$ is odd, then, for any sequence $\bm{i}=(i_1,i_2,\dots,i_{2k})$, 
$$
\bE [\Tilde{v}_{i_1,i_2} \Tilde{v}_{i_3,i_4} \cdots 
\Tilde{v}_{i_{2k-1}, i_{2k}}]  = 0.
$$
\end{thm}

\begin{proof}
The proof is similar to that of Theorem \ref{thm:CII}.
The Weingarten calculus for a Haar orthogonal matrix $R$ gives
$$
\bE [\Tilde{v}_{i_1,i_2} \Tilde{v}_{i_3,i_4} \cdots 
\Tilde{v}_{i_{2k-1}, i_{2k}}]  = 
\sum_{\sigma \in M_{2k}} \Delta_\sigma(\bm{i}) 
\sum_{\tau \in M_{2k}} \mathrm{Wg}^{\mathrm{O}}(\tau^{-1} \sigma;2n) \widetilde{T}(\tau)
$$
where
$$
\widetilde{T}(\tau)=\sum_{\bm{p}=(p_1,\dots,p_{2k})} \Delta_\tau(\bm{p})
\prod_{r=1}^k \langle p_{2r-1},p_{2r} \rangle 
= \mcal{T}_\tau(I,J) = T_n^{\mathrm{D \, III}}(\tau)
$$
for any $\tau \in S_{2k}$.
\end{proof}

Note that
$$
\mathrm{Wg}^{\mathrm{D \, III}}(\zeta\sigma\zeta';n)=
\epsilon(\zeta)\mathrm{Wg}^{\mathrm{D \, III}}(\sigma;n) \qquad (
\sigma\in S_{2k}, \ \zeta,\zeta' \in H_k),
$$
and hence $\mathrm{Wg}^{\mathrm{D \, III}}(\zeta;n)=0$ for $\zeta \in H_k$.
Furthermore, $\mathrm{Wg}^{\mathrm{D \, III}}(\sigma;n) =0 \ (\sigma \in S_{2k})$ if $k$ is odd.

\begin{example}
Since $\(\begin{smallmatrix} 1 & 2 & 3 & 4 \\ 1 & 3 & 2 & 4 \end{smallmatrix}\)
=(3 \ 4)\(\begin{smallmatrix} 1 & 2 & 3 & 4 \\ 1 & 4 & 2 & 3 \end{smallmatrix}\)$
and $(3 \ 4) \in H_2$, we have
$$
(-1)\mathrm{Wg}^{\mathrm{D \, III}}
\( \(\begin{smallmatrix} 1 & 2 & 3 & 4 \\ 1 & 3 & 2 & 4 \end{smallmatrix}\);n\)
=\mathrm{Wg}^{\mathrm{D \, III}}
\( \(\begin{smallmatrix} 1 & 2 & 3 & 4 \\ 1 & 4 & 2 & 3 \end{smallmatrix}\);n\)
= \frac{-1}{2n-1}.
$$
\end{example}

\subsection{Class C I}

{\it The setting for C  I}: 
$G=\rmSp(2n)$, $K=\rmU(n)$, 
$\Omega(g)= I'_{nn} g I'_{nn}$.
Here the fix-point set of $\Omega$ in $G$ is 
$$
\left\{ \begin{pmatrix}
U & O \\ O & \overline{U} \end{pmatrix}
 \ \bigm| \ U \in \rmU(n) 
\right\} \simeq \rmU(n).
$$

Let  $S$ be a $2n \times 2n$ Haar symplectic matrix.
We consider the random matrix 
$$
W^{\mathrm{C \, I}}= \dual{S} I'_{nn} S,
$$ 
instead of $V^{\mathrm{C \, I}}= \Omega(S)^{-1}S=I'_{nn} S^D I'_{nn} S$,
which is associated to the symmetric space $\mathrm{Sp}(2n)/\mathrm{U}(n)$.

We define the function $T^{\mathrm{C \, I}}_n$ on $S_{2k}$ by
$$
T^{\mathrm{C \, I}}_n(\sigma)=\mcal{T}_\sigma(J, JI'_{nn}).
$$
Since $JI'_{nn}= - \(\begin{smallmatrix} O_n& I_n \\ I_n & O_n \end{smallmatrix}\)$
is symmetric and $J$ is skew-symmetric,
the function  $T^{\mathrm{C \, I}}_n$ satisfies
$T^{\mathrm{C \, I}}_n(\zeta \sigma \zeta')=T^{\mathrm{C \, I}}_n(\sigma) \epsilon(\zeta')$ for $\sigma \in S_{2k}$ and $\zeta,\zeta' \in H_k$.
For $\sigma=\sigma_\mu$, we have 
\begin{align*}
T^{\mathrm{C \, I}}_n (\sigma_\mu)=& 
\prod_{i=1}^{\ell(\mu)} T^{\mathrm{C \, I}}_n(\sigma_{(\mu_i)})
= \prod_{i=1}^{\ell(\mu)} [-\Tr (JJI'_{nn} )^{\mu_i}] 
= (-1)^{\ell(\mu)} \prod_{i=1}^{\ell(\mu)}
\Tr (-I'_{nn})^{\mu_i}\\
=&
\begin{cases}
(-2n)^{\ell(\mu)} & \text{if $\mu$ is even},\\
0 & \text{otherwise}.
\end{cases}
\end{align*}
In particular,
$T^{\mathrm{C \, I}}_n (\sigma)=
T^{\mathrm{D \, III}}_n (\sigma^{-1})=\pm T^{\mathrm{D \, III}}_n (\sigma)$
for all $\sigma \in S_{2k}$.

\begin{thm}
Let $W= W^{\mathrm{C \, I}}$ be the random matrix defined as above.
For a sequence $\bm{i}=(i_1,i_2,\dots,i_{2k})$ with an even number $k$, we have
$$
\bE [ \Tilde{w}_{i_1,i_2} \Tilde{w}_{i_3,i_4} \cdots \Tilde{w}_{i_{2k-1}, i_{2k}}] = 
\sum_{\sigma \in M_{2k}} \Delta'_\sigma(\bm{i}) \mathrm{Wg}^{\mathrm{C \, I}}(\sigma;n).
$$
Here $\mathrm{Wg}^{\mathrm{C \, I}}(\cdot;n)$ is the function on $S_{2k}$
defined by
$$
\mathrm{Wg}^{\mathrm{C \, I}}(\cdot;n) =T^{\mathrm{C \, I}}_n \star \mathrm{Wg}^{\mathrm{Sp}}(\cdot;n).
$$ 
If $k$ is odd, then, for any sequence $\bm{i}=(i_1,i_2,\dots,i_{2k})$,
$$
\bE [ \Tilde{w}_{i_1,i_2} \Tilde{w}_{i_3,i_4} \cdots \Tilde{w}_{i_{2k-1}, i_{2k}}]  = 0.
$$
\end{thm}

\begin{proof}
It is proved in a usual way.
The Weingarten calculus for a Haar symplectic matrix $S$ gives
$$
\bE [ \Tilde{w}_{i_1,i_2} \Tilde{w}_{i_3,i_4} \cdots \Tilde{w}_{i_{2k-1}, i_{2k}}] = 
\sum_{\sigma \in M_{2k}} \Delta_\sigma'(\bm{i}) 
\sum_{\tau \in M_{2k}} \mathrm{Wg}^{\mathrm{Sp}}(\tau^{-1} \sigma;n) \widetilde{T}(\tau)
$$
where
$$
\widetilde{T}(\tau):=\sum_{\bm{p}=(p_1,\dots,p_{2k})} \Delta'_\tau(\bm{p})
\prod_{r=1}^k \langle e_{p_{2r-1}}, I'_{nn} e_{p_{2r}} \rangle 
= \mcal{T}_\tau(J, JI_{nn}') = T_n^{\mathrm{C \, I}}(\tau)
$$
for any $\tau \in S_{2k}$.
\end{proof}

The function $\mathrm{Wg}^{\mathrm{C \, I}}(\cdot;n)$ on $S_{2k}$ satisfies
$$
\mathrm{Wg}^{\mathrm{C \, I}}(\zeta\sigma\zeta';n)=
\epsilon(\zeta')\mathrm{Wg}^{\mathrm{C \, I}}(\sigma;n) \qquad (
\sigma\in S_{2k}, \ \zeta,\zeta' \in H_k),
$$
and hence $\mathrm{Wg}^{\mathrm{C \, I}}(\zeta;n)=0$ for $\zeta \in H_k$.
Furthermore, $\mathrm{Wg}^{\mathrm{C \, I}}(\sigma;n) =0 \ (\sigma \in S_{2k})$ if $k$ is odd.

\begin{example}
$$
\mathrm{Wg}^{\mathrm{C \, I}}
\( \(\begin{smallmatrix} 1 & 2 & 3 & 4 \\ 1 & 3 & 2 & 4 \end{smallmatrix}\);n\)
=\mathrm{Wg}^{\mathrm{C \, I}}
\( \(\begin{smallmatrix} 1 & 2 & 3 & 4 \\ 1 & 4 & 2 & 3 \end{smallmatrix}\);n\)
= \frac{-1}{2n+1}.
$$
\end{example}

\section{Conclusion}

We have made methods for computations of 
moments of matrix elements from classical compact Lie groups and
classical compact symmetric spaces.
Write $\mcal{C}=$ A, B/D, C for unitary, orthogonal, symplectic groups, respectively.
The moment for a classical group is given by the double sum
$$
\sum_{\sigma} \sum_{\tau} 
\text{($\Delta$-function in $\sigma$)} \times
\text{($\Delta$-function in $\tau$)} \times  
\mathrm{Wg}^{\mcal{C}} (\sigma^{-1} \tau;n),
$$
whereas
that for a classical compact symmetric space
is given by the single sum
$$
\sum_\sigma \text{($\Delta$-function in $\sigma$)} \times 
\mathrm{Wg}^{\mcal{C}} (\sigma;n).
$$
Here the $\Delta$-function is 
\begin{itemize}
 \item  $\delta_\sigma (\cdot,\cdot)$ defined in \eqref{eq:delta-ft}
  if $\mcal{C}=$ A, A~I, A~II, A~III,
  \item
  $\Delta_\sigma(\cdot)$ defined in \eqref{eq:Delta-ft} 
   if $\mcal{C}=$ B/D, BD~I, D~III, 
  \item
    $\Delta_\sigma'(\cdot)$ defined in \eqref{eq:Delta'-ft}
    if  $\mcal{C}=$ C, C~I, C~II,
\end{itemize}
and the Weingarten function $\mathrm{Wg}^{\mcal{C}}$ belongs to
\begin{itemize}
\item $\mcal{Z}(L(S_k))$ if $\mcal{C}=$ A, A~III,
\item $L(S_{2k},H_k)$ if $\mcal{C}=$ B/D, A~I, BD~I,
\item $L^\epsilon(S_{2k},H_k)$ if $\mcal{C}=$ C, A~II, C~II.
\end{itemize}
The Weingarten functions for $\mcal{C}=$ D~III or C~I differ from others.
In fact, if we let $F(\sigma)$ to be $\mathrm{Wg}^{\mathrm{D \, III}}(\sigma;n)$ or
$\mathrm{Wg}^{\mathrm{C \, I}}(\sigma^{-1};n)$, then
$F$ is the function on $S_{2k}$ satisfying the property
$$
F(\zeta \sigma \zeta')= \epsilon(\zeta) F( \sigma) 
\qquad (\sigma \in S_{2k}, \ \zeta,\zeta' \in H_k).
$$
In particular, $F$ identically vanishes  on $H_k$, and, moreover,
on $S_{2k}$ if $k$ is odd.
Weingarten functions except D~III and C~I have Fourier expansions in
$\chi^\lambda$, $\omega^\lambda$, or $\pi^\lambda$. 
However, we could not find such expansions for these two cases.


\section*{Acknowledgments}

The author's work was 
supported by JSPS Grant-in-Aid for Young Scientists (B) 22740060.


\bigskip

\noindent
\textsc{Sho Matsumoto} \\
Graduate School of Mathematics, Nagoya University, Nagoya, 464-8602, Japan. \\
\verb|sho-matsumoto@math.nagoya-u.ac.jp|

\end{document}